\documentclass[11pt]{article}
\usepackage[utf8]{inputenc}
\usepackage[left=3.5cm,top=3.5cm,right=3.5cm,bottom=3.5cm]{geometry} 
\usepackage{amsthm,bbm,amsfonts,amsmath,dsfont,appendix,graphicx,tikz}
\usepackage[pdfborder={0 0 0}]{hyperref}
\usetikzlibrary{matrix,arrows}

\numberwithin{equation}{section}
\newcounter{dummy} \numberwithin{dummy}{section}

\newtheorem{Teorema}[dummy]{Theorem}

\newtheorem{Lema}[dummy]{Lemma}

\newcommand{\Fc}{\mathcal{F}}
\newcommand{\Sc}{\mathcal{S}}

\newcommand{\Hc}{\mathcal{H}}

\newcommand{\Rc}{\mathcal{R}}

\newcommand{\Tc}{\mathcal{T}}

\newcommand{\N}{\mathbb{N}}

\newcommand{\R}{\mathbb{R}}

\newcommand{\E}{\mathbb{E}}

\newcommand{\Pb}{\mathbb{P}}

\newcommand{\Hg}{\mathfrak{H}}

\newcommand{\Ind}[1]{\mathbbm{1}_{\{#1\}}}
\newcommand{\Indi}[1]{\mathbbm{1}_{#1}}

\newcommand{\Ip}[1]{\left\langle #1 \right\rangle}

\newcommand{\Norm}[1]{\left\lVert#1\right\rVert}
\newcommand{\Abs}[1]{\left|#1\right|}

\begin{document}
\title{Asymptotic properties of the derivative of self-intersection local time of fractional Brownian motion}
\author{ \sc {Arturo Jaramillo and David Nualart}\thanks{D. Nualart is supported by the NSF grant DMS1512891 and the ARO grant FED0070445} 
\\ Department of Mathematics \\The University of Kansas \\ Lawrence, Kansas, 66045}
\date{}
\maketitle
\begin{abstract}
Let $\{B_{t}\}_{t\geq0}$ be a fractional Brownian motion with Hurst parameter $\frac{2}{3}<H<1$. We prove that the approximation of the derivative of self-intersection local time, defined as 
\begin{align*}
\alpha_{\varepsilon}
  &=\int_{0}^{T}\int_{0}^{t}p'_{\varepsilon}(B_{t}-B_{s})\text{d}s\text{d}t,
\end{align*}  
where $p_\varepsilon(x)$ is the heat kernel, satisfies a central limit theorem   when renormalized by $\varepsilon^{\frac{3}{2}-\frac{1}{H}}$. We prove as well that for $q\geq2$, the $q$-th chaotic component of $\alpha_{\varepsilon}$ converges in $L^{2}$ when $\frac{2}{3}<H<\frac{3}{4}$, and satisfies a central limit theorem when renormalized by a multiplicative factor $\varepsilon^{1-\frac{3}{4H}}$ in the case $\frac{3}{4}<H<\frac{4q-3}{4q-2}$.
\end{abstract}

\section{Introduction}

Let $B=\{B_{t}\}_{t\geq0}$ be a one-dimensional fractional Brownian motion of Hurst parameter $H\in(0,1)$. Fix $T>0$. 
The self-intersection local time of $B$, formally defined by
\[
I(y):=\int_0^T \int_0^t \delta(B_t-B_s-y) \text{d} s \text{d}t,
\]
was first studied by Rosen in \cite{Ro} in the planar case and it was further investigated using techniques from Malliavin calculus by Hu and Nualart  in  \cite{HuNu}. In particular, in \cite{HuNu} it is proved that for a $d$-dimensional fractional Brownian motion, $I(0)$ exists in $L^2$  whenever the Hurst parameter $H$ satisfies $H<\frac 1d$.

Motivated by spatial integrals with respect to local time, developed by Rogers and Walsh in  \cite{RoWa},    Rosen introduced  in \cite{RDSIL}  a formal derivative of $I(y)$, in the one-dimensional Brownian case, denoted by
\[
\alpha(y):= \frac {\text{d} I}{\text{d}y} (y)=-\int_0^T \int_0^t \delta'(B_t-B_s-y) \text{d} s \text{d}t.
\]
The random variable $\alpha:=\alpha(0)$ is called the derivative of the self-intersection local time at zero, and is equal to the limit in $L^2$ of
\begin{align}
\alpha_{\varepsilon}    \label{eq1:30/092015}
  &:=\int_{0}^{T}\int_{0}^{t}p'_{\varepsilon}(B_{t}-B_{s})\text{d}s\text{d}t,
\end{align}
where $p_\varepsilon(x):=(2\pi \varepsilon )^{-\frac 12} e^{-\frac {x^2}{2\varepsilon}}$.
This random variable was subsequently used by Hu and Nualart \cite{HuNu2}, to study the asymptotic properties of the third spacial moment of the Brownian local time. In \cite{MDILTCHAOS}, Markowsky gave an alternative proof the existence of such limit by using  Wiener chaos expansion.

 Jung and Markowsky extended this result  in  \cite{JMTANAKA} to the case $0<H<\frac{2}{3}$ and conjectured that for the case $H>\frac{2}{3}$, $\varepsilon^{-\gamma(H)}\alpha_{\varepsilon}$ converges in law to a Gaussian distribution for some suitable constant $\gamma(H)>0$, and at the critical point $H=\frac{2}{3}$, the variable $\log(1/\varepsilon)^{-\gamma}\alpha_{\varepsilon}$ converges in law to a Gaussian distribution for some $\gamma>0$. 
	
Let $\mathcal{N}(0,\gamma)$ denote a centered Gaussian random variable with variance $\gamma$. The primary goal of this paper is to analyze the behavior of the law of $\alpha_{\varepsilon}$ as $\varepsilon\rightarrow0$, when $\frac{2}{3}<H<1$. We will prove that when $\frac{2}{3}<H<1$,  
\begin{align*}
\varepsilon^{\frac{3}{2}-\frac{1}{H}}\alpha_{\varepsilon}
  &\stackrel{Law}{\rightarrow}\mathcal{N}(0,\sigma^{2}),~~~~~~\text{when}~~\varepsilon\rightarrow0,
\end{align*}
for some constant $\sigma^{2}$ that will be specified later (see Theorem \ref{theorem:l02limit/02/04}).  Moreover, we will prove that for every $q\geq2$ and $\frac{2}{3}<H<\frac{3}{4}$, $\lim_{\varepsilon\rightarrow0}J_{q}\left[\alpha_{\varepsilon}\right]$ exists in $L^{2}$, where $J_{q}$ denotes the projection on the $q$-th Wiener chaos (see Theorem \ref{theorem:l2limit}), while in the case $\frac{3}{4}<H<\frac{4q-3}{4q-2}$, the chaotic components $J_{q}\left[\alpha_{\varepsilon}\right]$ of $\alpha_{\varepsilon}$ satisfy 
\begin{align*}
\varepsilon^{1-\frac{3}{4H}}J_{q}\left[\alpha_{\varepsilon}\right]
  &\stackrel{Law}{\rightarrow}\mathcal{N}(0,\sigma_{q}^{2}),~~~~~~\text{when}~~\varepsilon\rightarrow0,
\end{align*}
for some constant $\sigma_{q}^{2}$ that will be specified latter (see Theorem \ref{theorem:distributionchaos}). The proof of the central limit theorem for  $\varepsilon^{\frac{3}{2}-\frac{1}{H}}\alpha_{\varepsilon}$ follows easily from estimations of the $L^{2}$-norm of the chaotic components of $\alpha_{\varepsilon}$, while the proof of the central limit theorem for $\varepsilon^{1-\frac{3}{4H}}J_{q}\left[\alpha_{\varepsilon}\right]$ relies on the multivariate version of the fourth moment theorem (see \cite{NP,PTGLVVMSI}), as well as the a continuos version of the Breuer-Major theorem (\cite{BM}) proved in \cite{NNDLTVolt}.
The behavior of $\alpha_\varepsilon$ in the critical case $H=\frac 23$, and the behavior of  $J_q[\alpha_{\varepsilon}]$ in the critical cases $H=\frac 23$, $H=\frac 34$ and $H=\frac{4q-3}{4q-2}$ seems more involved and will not be discussed in this paper.

It is surprising to remark that the limit behavior of the chaotic components of $\alpha_\varepsilon$ is different from that of the whole sequence. This phenomenon was observed, for instance, in the central  limit theorem for the second spatial moment of  Brownian local time  increments (see \cite{DNT}). However, in this case  the limit of the whole sequence is a mixture of Gaussian distributions, whereas in the present paper  the normalization of $\alpha_\varepsilon$ converges to a Gaussian law. In our case, the projection on the first chaos  of $ \alpha_\varepsilon$ is the leading term and is responsible for the Gaussian limit  of the whole sequence.

The paper is organized as follows. In Section \ref{sec:chaos} we present some preliminary results on the fractional Brownian motion and the chaotic decomposition of $\alpha_{\varepsilon}$. In Section \ref{section:variances} we compute the asymptotic behavior of the variances of the  normalizations of the chaotic components of $\alpha_{\varepsilon}$ as $\varepsilon\rightarrow0$. The asymptotic behavior of the law of $\alpha_{\varepsilon}$ and its chaotic components is presented in section \ref{section:limit_theorems}. Finally, some technical lemmas are proved in Section 5.
\section{Preliminaries}\label{sec:chaos}
\subsection{Fractional Brownian motion}
Throughout the paper, $B=\{B_{t}\}_{t\geq0}$ will denote a fractional Brownian motion with Hurst parameter $H\in(0,1)$, defined on a probability space $(\Omega,\Fc,\Pb)$. That is, $B$ is a centered Gaussian process with covariance function 
\begin{align*}
\E\left[B_{t}B_{s}\right]
  &=\frac{1}{2}(t^{2H}+s^{2H}-|t-s|^{2H}).
\end{align*}
We will denote by $\Hg$ the Hilbert space obtained by taking the completion of the space of step functions endowed with the inner product 
\begin{align*}
\Ip{\Indi{[a,b]},\Indi{[c,d]}}_{\Hg}
  &:=  \E\left[(B_{b}-B_{a})(B_{d}-B_{c})\right].
\end{align*}
The mapping $\Indi{[0,t]} \mapsto B_t$ can be extended to a linear isometry between $\Hg$ and a Gaussian subspace of $L^2(\Omega, \Fc, \Pb)$. We will denote by $B(h)$ the image of $h\in \Hg$ by this isometry.
For any integer $q\in\N$, we denote by $\Hg^{\otimes q}$ and $\Hg^{\odot q}$ the $q$th tensor product of $\Hg$, and the $q$th symmetric tensor product of $\Hg$, respectively. The $q$th Wiener chaos	of $L^{2}(\Omega,\Fc,\Pb)$, denoted by $\Hc_{q}$, is the closed subspace of $L^{2}(\Omega,\Fc,\Pb)$ generated by the variables $\{H_{q}(B(h)),~  h\in\Hg,\Norm{h}_{\Hg}=1\}$, where $H_{q}$ is the $q$th Hermite polynomal, defined by 
\begin{align*}
H_{q}(x)
  &:=(-1)^{q}e^{\frac{x^{2}}{2}}\frac{\text{d}^{p}}{\text{d}x^{p}}e^{-\frac{x^{2}}{2}}.
\end{align*}
The mapping $I_{q}(h^{\otimes q})=H_{q}(B(h))$ provides a linear isometry between $\Hg^{\odot q}$ (equipped with the norm $\sqrt{q}!\Norm{\cdot}_{\Hg^{\otimes q}}$) and $\Hc_{q}$ (equipped with the $L^{2}$-norm).
\subsection{Chaos decomposition for \texorpdfstring{$\alpha_{\varepsilon}$}{Text}}

Proceeding as in \cite{JMTANAKA} (also see \cite{HuNu}), we can determine the chaos decomposition of the random variable $\alpha_{\varepsilon}$ defined in \eqref{eq1:30/092015} as follows.  First we write
\begin{align}\label{eq:alpha}
\alpha_\varepsilon= \int_0^T\int_0^t  \alpha_{\varepsilon,s,t} {\text d} s {\text d}t,
\end{align}
where $\alpha_{\varepsilon,s,t}:= p_{\varepsilon}^{\prime}(B_{t}-B_{s})$. 
 We know that
\begin{align}\label{eq5m:08/08/2015}
\alpha_{\varepsilon,s,t}	
  &= \sum_{q=1}^{\infty}I_{2q-1 }\left(f_{2q-1,\varepsilon,s,t}\right),
\end{align}
where 
\begin{align}\label{eq:kernel_definition}
f_{2q-1,\varepsilon,s,t}(x_{1},\dots,x_{2q-1})
 &:= (-1)^{q}\beta_{q}(\varepsilon+(t-s)^{2H})^{-q-\frac{1}{2}}\prod_{j=1}^{2q-1}\Indi{[s,t]}(x_{j}),
\end{align}
and
\begin{align}\label{betadef}
\beta_{q}
  :=\frac{1}{2^{q-\frac{1}{2}}(q-1)!\sqrt{\pi}}.
\end{align}
As a consequence, the random variable $\alpha_{\varepsilon}$ has 	the chaos decomposition
\begin{align}\label{chaos:alpha}
\alpha_{\varepsilon}
  &=\sum_{q=1}^{\infty}I_{2q-1 }(f_{2q-1,\varepsilon}),
\end{align}
where 
\begin{align}\label{fdef27/01}
f_{2q-1,\varepsilon}(x_{1},\dots,x_{2q-1})
 &:= \int_{\Rc}f_{2q-1,\varepsilon,s,t}(x_{1},\dots,x_{2q-1})\text{d}s\text{d}t,
\end{align}
and 
\begin{align}\label{fdef27asdn/01}
\Rc:=\{(s,t)\in\R_{+}^{2}\ |\ s\leq t\leq T\}.
\end{align}
Let $T,\varepsilon>0$, $\frac{2}{3}<H<1$, and $q\in\N$ be fixed. Our first goal is to find the behavior as $\varepsilon\rightarrow0$ of the variances of $\alpha_{\varepsilon}$ and $I_{2q-1}\left(f_{2q-1,\varepsilon}\right)$. Before addressing this problem, we will  introduce some notation.
  First notice that
\begin{align}
\E\left[I_{2q-1}\left(f_{2q-1,\varepsilon}\right)^{2}\right]
  &= (2q-1)!\Norm{f_{2q-1,\varepsilon}}_{\Hg^{\otimes(2q-1)}}^{2}\nonumber\\
	&= (2q-1)!\Ip{\int_{\Rc}f_{2q-1,\varepsilon,s_{1},t_{1}}\text{d}s_{1}\text{d}t_{1},\int_{\Rc}f_{2q-1,\varepsilon,s_{2},t_{2}}\text{d}s_{2}\text{d}t_{2}}_{\Hg^{\otimes(2q-1)}}\nonumber\\
	&= 2(2q-1)!\int_{\Sc}\Ip{f_{2q-1,\varepsilon,s_{1},t_{1}},f_{2q-1,\varepsilon,s_{2},t_{2}}}_{\Hg^{\otimes(2q-1)}}\text{d}s_{1}\text{d}s_{2}\text{d}t_{1}\text{d}t_{2}\label{eq1a:06/08/2015},
\end{align}
where the set  $\Sc$ is defined by 
\begin{align}
\Sc
  &:=\{(s_{1},s_{2},t_{1},t_{2})\in[0,T]^{4}\ |\ s_{1}\leq t_{1},~~s_{2}\leq t_{2},~~\text{ and }~~s_{1}\leq s_{2}\}\label{eq:1:regions}.
\end{align}
We can write the set $\Sc$ as the union of the sets $\Sc_{1},\Sc_{2}, \Sc_{3}$ defined by 
\begin{align}
\Sc_{1}
  &:=\{(s_{1},s_{2},t_{1},t_{2})\in[0,T]^{4}\ |\ s_{1}\leq s_{2}\leq t_{1}\leq t_{2}\}\label{eq:2:regions},\\
\Sc_{2}
  &:=\{(s_{1},s_{2},t_{1},t_{2})\in[0,T]^{4}\ |\ s_{1}\leq s_{2}\leq t_{2}\leq t_{1}\}\label{eq:3:regions},\\
\Sc_{3}
  &:=\{(s_{1},s_{2},t_{1},t_{2})\in[0,T]^{4}\ |\ s_{1}\leq t_{1}\leq s_{2}\leq t_{2}\}\label{eq:4:regions}.
\end{align}
Then, by \eqref{eq:alpha},
\begin{align}\label{eq1n:06/08/2015}
\E\left[\alpha_{\varepsilon}^{2}\right]
  &=\E\left[\left(\int_{\Rc}\alpha_{\varepsilon, s, t}\text{d}s\text{d}t\right)^{2}\right]\nonumber\\
	&=2\int_{\Sc}\E\left[\alpha_{\varepsilon,s_{1},t_{1}}\alpha_{\varepsilon,s_{2},t_{2}}\right]\text{d}s_{1}\text{d}s_{2}\text{d}t_{1}\text{d}t_{2}\nonumber\\
  &= V_{1}(\varepsilon)+V_{2}(\varepsilon)+V_{3}(\varepsilon),
\end{align}
where 
\begin{align}\label{Vdef}
V_{i}(\varepsilon)
  &:=2\int_{\Sc_{i}}\E\left[\alpha_{\varepsilon,s_{1},t_{1}}\alpha_{\varepsilon,s_{2},t_{2}}\right]\text{d}s_{1}\text{d}s_{2}\text{d}t_{1}\text{d}t_{2},~~~~~~~~~~~i=1,2,3.
\end{align}
 Similarly, from \eqref{fdef27/01} and  \eqref{eq1a:06/08/2015}, taking $\varepsilon=1$, we get  
\begin{align}\label{eq4a:06/08/2015}
\E\left[I_{1}\left(f_{1,\varepsilon}\right)^{2}\right]
  &= V_{1}^{(1)}(\varepsilon)+V_{2}^{(1)}(\varepsilon)+V_{3}^{(1)}(\varepsilon),
\end{align}
where 
\begin{align}\label{Vqdef}
V_{i}^{(1)}(\varepsilon)
  &:=2\int_{\Sc_{i}}\Ip{f_{1,\varepsilon,s_{1},t_{1}},f_{1,\varepsilon,s_{2},t_{2}}}_{\Hg}\text{d}s_{1}\text{d}s_{2}\text{d}t_{1}\text{d}t_{2},~~~~~~~~~~i=1,2,3.
\end{align}
As a consequence of \eqref{eq1n:06/08/2015} and \eqref{eq4a:06/08/2015}, to determine the behavior of the variances of $\alpha_{\varepsilon}$ and $I_{1}\left(f_{1,\varepsilon}\right)$ as $\varepsilon\rightarrow0$, it suffices to determine the behavior of $V_{i}(\varepsilon)$ and $V_{i}^{(1)}(\varepsilon)$ respectively, for $i=1,2,3$.

In order to describe the terms $\Ip{f_{2q-1,\varepsilon,s_{1},t_{1}},f_{2q-1,\varepsilon,s_{2},t_{2}}}_{\Hg^{\otimes(2q-1)}}$, we will introduce the following notation. For every $x,u_{1},u_{2}>0$ define
\begin{align}\label{mudef}
\mu(x,u_{1},u_{2})
  &:= \E\left[B_{u_{1}}(B_{x+u_{2}}-B_{x})\right].
\end{align}
We can easily prove that for every $s_{1},s_{2},t_{1},t_{2}\geq0$, such that $s_{1}\leq t_{1}$, $s_{2}\leq  t_{2}$ and $s_{1}\leq s_{2}$,	 
\begin{align}\label{muprop}
\E\left[(B_{t_{1}}-B_{s_{1}})(B_{t_{2}}-B_{s_{2}})\right]
  &=\mu(s_{2}-s_{1},t_{1}-s_{1},t_{2}-s_{2}).
\end{align}
Using \eqref{eq:kernel_definition} and \eqref{muprop}, for every $0\leq s_{1}\leq t_{1}$, $0\leq s_{2}\leq t_{2}$ such that $s_{1}\leq s_{2}$, we can write
\begin{align*}
\Ip{f_{2q-1,\varepsilon,s_{1},t_{1}},f_{2q-1,\varepsilon,s_{2},t_{2}}}_{\Hg^{\otimes(2q-1)}}
  &=\beta_{q}^{2}(\varepsilon+(t_{1}-s_{1})^{2H})^{-\frac{1}{2}-q}(\varepsilon+(t_{2}-s_{2})^{2H})^{-\frac{1}{2}-q}\\
	&~~\times\Ip{\Indi{[s_{1},t_{1}]}^{\otimes(2q-1)},\Indi{[s_{2},t_{2}]}^{\otimes(2q-1)}}_{\Hg^{\otimes(2q-1)}}\\
	&=\beta_{q}^{2}(\varepsilon+(t_{1}-s_{1})^{2H})^{-\frac{1}{2}-q}(\varepsilon+(t_{2}-s_{2})^{2H})^{-\frac{1}{2}-q}\\
	&~~\times\mu(s_{2}-s_{1},t_{1}-s_{1},t_{2}-s_{2})^{2q-1}.
\end{align*}
Therefore,  
\begin{align}
\Ip{f_{2q-1,\varepsilon,s_{1},t_{1}},f_{2q-1,\varepsilon,s_{2},t_{2}}}_{\Hg^{\otimes(2q-1)}}
  &=\beta_{q}^{2}G_{\varepsilon,s_{2}-s_{1}}^{(q)}(t_{1}-s_{1},t_{2}-s_{2})\label{eq2a:06/08/2015},
\end{align}
where $G_{\varepsilon,x}^{(q)}(u_{1},u_{2})$ is defined by
\begin{align}\label{Gdef}
G_{\varepsilon,x}^{(q)}(u_{1},u_{2})
  &:=\left(\varepsilon+u_{1}^{2H}\right)^{-\frac{1}{2}-q}\left(\varepsilon+u_{2}^{2H}\right)^{-\frac{1}{2}-q}\mu(x,u_{1},u_{2})^{2q-1}.
\end{align}
Next we present some useful properties of the functions $\mu(x,u_{1},u_{2})$ and $G_{\varepsilon,x}^{(q)}(u_{1},u_{2})$. Taking into account that $H>\frac{2}{3}$, we can write the covariance of $B$ as 
\begin{align}\label{covrep}
\E\left[B_{t}B_{s}\right]
  &=H(2H-1)\int_{0}^{t}\int_{0}^{s}\left|v_{1}-v_{2}\right|^{2H-2}\text{d}v_{1}\text{d}v_{2}.
\end{align}
In particular, this leads to 
\begin{align}\label{eq100ma:16/08/2015}
\mu(x,u_{1},u_{2})
  =H(2H-1)\int_{0}^{u_{1}}\int_{x}^{x+u_{2}}\left|v_{2}-v_{1}\right|^{2H-2}\text{d}v_{1}\text{d}v_{2},
\end{align}
which implies 
\begin{align}\label{Ggeq0}
G_{\varepsilon,x}^{(q)}(u_{1},u_{2})\geq0~~~~~~ \text{for every}~~ \varepsilon\geq0.
\end{align} 
Using the chaos decomposition \eqref{eq5m:08/08/2015}, as well as \eqref{eq2a:06/08/2015} and \eqref{Ggeq0}, we can check that for $i=1,2,3$, the terms $V_{i}(\varepsilon), V_{i}^{(1)}(\varepsilon)$, defined by \eqref{Vdef}, \eqref{Vqdef}, satisfy 
\begin{align}\label{eq1ma:26/08/2015}
0\leq V_{i}^{(1)}(\varepsilon)\leq V_{i}(\varepsilon).
\end{align}
Further properties for the function $G_{\varepsilon,x}^{(q)}(u_{1},u_{2})$ are described in the following lemma.
\begin{Lema}\label{remarkG}		
Let $G^{(q)}_{1,x}(u_{1},u_{2})$ be defined by \eqref{Gdef}. There exists a constant $K>0$, depending on $H$ and $q$, such that for all $x>0$, and $0<v_{1}\leq w_{1}$, $0<v_{2}\leq w_{2}$ satisfying $|v_{i}-w_{i}|\leq1$,
\begin{align*}
G_{1,x}^{(q)}(v_{1},v_{2})
  &\leq KG_{1,x}^{(q)}(w_{1},w_{2}).
\end{align*}
\end{Lema}
\begin{proof}
From \eqref{eq100ma:16/08/2015} it follows that 
\begin{align*}
\mu(x,v_{1},v_{2})
  &\leq \mu(x,w_{1},w_{2}).
\end{align*}
As a consequence, 
\begin{align*}
G_{1,x}^{(q)}(v_{1},v_{2})
  &=(1+v_{1}^{2H})^{-\frac{1}{2}-q}(1+v_{2}^{2H})^{-\frac{1}{2}-q}\mu(x,v_{1},v_{2})^{2q-1}\\
	&\leq(1+v_{1}^{2H})^{-\frac{1}{2}-q}(1+v_{2}^{2H})^{-\frac{1}{2}-q}\mu(x,w_{1},w_{2})^{2q-1}\\
	&=G_{1,x}^{(q)}(w_{1},w_{2})\left(\frac{(1+w_{1}^{2H})(1+w_{2}^{2H})}{(1+v_{1}^{2H})(1+v_{2}^{2H})}\right)^{q+\frac{1}{2}}.
\end{align*}
Using condition $|v_{i}-w_{i}|\leq 1$, $i=1,2$, we get 
\begin{align}\label{eq2n:01/11/2015}
G_{1,x}^{(q)}(v_{1},v_{2})
  &\leq G_{1,x}^{(q)}(w_{1},w_{2})\left(\frac{(1+(v_{1}+1)^{2H})(1+(v_{2}+1)^{2H})}{(1+v_{1}^{2H})(1+v_{2}^{2H})}\right)^{q+\frac{1}{2}}.
\end{align}
The second factor in the right-hand side of \eqref{eq2n:01/11/2015} is uniformly bounded for $v_{1},v_{2}\geq0$, which implies the desired result.
\end{proof}

\section{Behavior of the variances of \texorpdfstring{$\alpha_{\varepsilon}$}{TEXT}  and its chaotic components}\label{section:variances}
The behavior of the variance of $\alpha_{\varepsilon}$ is described in the following lemma.
\begin{Lema}\label{theorem:variance}
Let $T>0$ and $\frac{2}{3}<H<1$ be fixed. Then,
\begin{align}
\lim_{\varepsilon\rightarrow0}\varepsilon^{3-\frac{2}{H}}\E\left[\alpha_{\varepsilon}^{2}\right]
  =\sigma^{2}\label{eq1n:20/07/2015},
\end{align} 
where $\sigma^{2}$ is defined by
\begin{align}\label{eq:5:01/04/2015}
\sigma^{2}
  &:=\frac{T^{2H}(2H-1)}{4H\pi}B\left(\frac{1}{H},\frac{3H-2}{2H}\right)^2B(2,2H-1),
\end{align}
and	 $B(\cdot,\cdot)$ denotes the Beta function. 
\end{Lema}
\begin{proof}
From \eqref{eq1n:06/08/2015} we have
\[
\varepsilon^{3-\frac{2}{H}}\E\left[\alpha_{\varepsilon}^{2}\right]
  = \varepsilon^{3-\frac{2}{H}}V_{1}(\varepsilon)+\varepsilon^{3-\frac{2}{H}}V_{2}(\varepsilon)+\varepsilon^{3-\frac{2}{H}}V_{3}(\varepsilon),
\]
where $V_{1}(\varepsilon)$, $V_{2}(\varepsilon)$  and $V_{3}(\varepsilon)$ are defined by \eqref{Vdef}. By Lemmas \ref{Lema:aux_region1c} and \ref{Lema1:aux_region2},  we have $\lim_{\varepsilon\rightarrow0}\varepsilon^{3-\frac{2}{H}}V_{1}(\varepsilon)=0$ and $\varepsilon^{3-\frac{2}{H}}V_{2}(\varepsilon)=0$, respectively. In addition, from Lemma \ref{Lema1:aux_region3:1/12/08/2015} we have 
$\lim_{\varepsilon\rightarrow0}\varepsilon^{3-\frac{2}{H}}V_{3}(\varepsilon)
  =\sigma^{2}$,  where $\sigma^{2}$ is defined by \eqref{eq:5:01/04/2015}.  This completes the proof of equation \eqref{eq1n:20/07/2015}.
\end{proof}
The behavior of the variance of the first chaotic component of $\alpha_{\varepsilon}$ is described by the following lemma.
\begin{Lema}\label{lemma:varianza1:28/04}	
Let $T>0$ be fixed. Define $f_{1,\varepsilon}$ as in equation \eqref{fdef27/01}. Then, for every $\frac{2}{3}<H<1$, we have
\begin{align}\label{eq87bt1ecgb:16/08/2015}
\lim_{\varepsilon\rightarrow0}\varepsilon^{3-\frac{2}{H}}\E\left[I_{1}\left(f_{1,\varepsilon}\right)^{2}\right]
  &=\sigma^{2},
\end{align}
where $\sigma^{2}$ is given by \eqref{eq:5:01/04/2015}.
\end{Lema}
\begin{proof}
From \eqref{eq4a:06/08/2015} we have
\[
\varepsilon^{3-\frac{2}{H}}\E\left[I_{1}(f_{1,\varepsilon})^{2}\right]
  = \varepsilon^{3-\frac{2}{H}}V_{1}^{(1)}(\varepsilon)+\varepsilon^{3-\frac{2}{H}}V_{2}^{(1)}(\varepsilon)+\varepsilon^{3-\frac{2}{H}}V_{3}^{(1)}(\varepsilon),
\]
where $V_{1}^{(1)}(\varepsilon)$,   $V_{2}^{(1)}(\varepsilon)$   and $V_{3}^{(1)}(\varepsilon)$ are defined by \eqref{Vqdef}. By Lemmas \ref{Lema:aux_region1c} and \ref{Lema1:aux_region2},  we have $\lim_{\varepsilon\rightarrow0}\varepsilon^{3-\frac{2}{H}}V_{1}(\varepsilon)=0$ and $\varepsilon^{3-\frac{2}{H}}V_{2}(\varepsilon)=0$, respectively.  Consequently, by \eqref{eq1ma:26/08/2015} we get 
$\lim_{\varepsilon\rightarrow0}\varepsilon^{3-\frac{2}{H}}V_{1}^{(1)}(\varepsilon)
  =0$ and $\lim_{\varepsilon\rightarrow0}\varepsilon^{3-\frac{2}{H}}V_{2}^{(1)}(\varepsilon)=0$.
In addition, from Lemma \ref{Lema2:aux_region3},  the term  $V_{3}^{(1)}(\varepsilon)$ satisfies
$\lim_{\varepsilon\rightarrow0}\varepsilon^{3-\frac{2}{H}}V_{3}^{(1)}(\varepsilon)
  =\sigma^{2}$,
where $\sigma^{2}$ is given by \eqref{eq:5:01/04/2015}. This completes the proof of equation  \eqref{eq87bt1ecgb:16/08/2015}.
\end{proof}
The behavior of the variance of the chaotic components of $\alpha_{\varepsilon}$ of order greater or equal to two and is described by the following lemma.
\begin{Lema}\label{lemma:varianza:28/04}	
Let $T,\varepsilon>0$, $\frac{2}{3}<H<1$ and $q\in\N$, $q\geq2$ be fixed. Define $\beta_{q},f_{2q-1,\varepsilon},$ and  $G_{\varepsilon,x}^{(q)}(u_{1},u_{2})$ by \eqref{betadef}, \eqref{fdef27/01} and \eqref{Gdef} respectively.  Then, 
\begin{enumerate}
\item  If $\frac{3}{4}<H<\frac{4q-3}{4q-2}$, 
\begin{align}
\lim_{\varepsilon\rightarrow0}\varepsilon^{2-\frac{3}{2H}}\E\left[I_{2q-1}(f_{2q-1,\varepsilon})^2\right]
   &=\sigma_{q}^{2}\label{eq5:20/07/2015},
\end{align}
where $\sigma_{q}^2$ is a finite constant given by
\begin{align}
\sigma_{q}^{2}
   &:=2(2q-1)!\beta_{q}^{2}T\int_{\R_{+}^{3}}G_{1,x}^{(q)}(u_{1},u_{2})\text{d}x\text{d}u_{1}\text{d}u_{2}\label{eq6:20/07/2015}.
\end{align}
\item  In the case $\frac{2}{3}<H<\frac{3}{4}$, 
\begin{align}
\lim_{\varepsilon\rightarrow0}\E\left[I_{2q-1}(f_{2q-1,\varepsilon})^2\right]
   &=\overline{\sigma}_{q}^{2}\label{eq10n:20/07/2015},
\end{align}
where $\overline{\sigma}_{q,d}^2$ is a finite constant given by
\begin{align}
\overline{\sigma}_{q,d}^{2}
   &:=2(2q-1)!\beta_{q}^{2}\int_{\Sc}G_{0,s_{2}-s_{1}}^{(q)}(t_{1}-s_{1},t_{2}-s_{2})\text{d}s_{1}\text{d}s_{2}\text{d}t_{1}\text{d}t_{2},\label{eq11n:20/07/2015}
\end{align}
and $\Sc$ is defined by \eqref{eq:1:regions}.
\end{enumerate}
\end{Lema}
\begin{proof}
First we prove \eqref{eq5:20/07/2015} in the case $\frac{3}{4}<H<\frac{4q-3}{4q-2}$.  By \eqref{eq1a:06/08/2015} and \eqref{eq2a:06/08/2015}, 
\begin{align*}
\varepsilon^{2-\frac{3}{2H}}\E\left[I_{2q-1}(f_{2q-1,\varepsilon})^2\right]
  &=2(2q-1)!\beta_{q}^{2}\varepsilon^{2-\frac{3}{2H}}\int_{\Sc}G_{\varepsilon,s_{2}-s_{1}}^{(q)}(t_{1}-s_{1},t_{2}-s_{2})\text{d}s_{1}\text{d}s_{2}\text{d}t_{1}\text{d}t_{2},
\end{align*}
where $\Sc$ is defined by \eqref{eq:1:regions}. Therefore, changing the coordinates $(s_{1},s_{2},t_{1},t_{2})$ by $(\varepsilon^{-\frac{1}{2H}}s_{1},x:=\varepsilon^{-\frac{1}{2H}}(s_{2}-s_{1}),u_{1}:=\varepsilon^{-\frac{1}{2H}}(t_{1}-s_{1}),u_{2}:=\varepsilon^{-\frac{1}{2H}}(t_{2}-s_{2}))$, we get 
\begin{align*}
\varepsilon^{2-\frac{3}{2H}}\E\left[I_{2q-1}(f_{2q-1,\varepsilon})^2\right]
	&=2(2q-1)!\beta_{q}^{2}\varepsilon^{\frac{1}{2H}}\int_{\R_{+}^{4}}\Indi{(0,\varepsilon^{-\frac{1}{2H}}T)}(s_{1}+u_{1})\\
	&\times\Indi{(0,\varepsilon^{-\frac{1}{2H}}T)}(s_{1}+x+u_{2})G_{1,x}^{(q)}(u_{1},u_{2})\text{d}s_{1}\text{d}x\text{d}u_{1}\text{d}u_{2}.
\end{align*}
Integrating with respect to the variable $s_{1}$ we get 
\begin{align}\label{eq3n:14/10/2015}
\varepsilon^{2-\frac{3}{2H}}\E\left[I_{2q-1}(f_{2q-1,\varepsilon})^2\right]
	&=2(2q-1)!\beta_{q}^{2}\int_{\R_{+}^{3}}(T-\varepsilon^{\frac{1}{2H}}(u_{1}\vee (x+u_{2})))\Indi{(0,\varepsilon^{-\frac{1}{2H}}T)}(u_{1})\nonumber\\
	&\times\Indi{(0,\varepsilon^{-\frac{1}{2H}}T)}(s_{1}+x+u_{2})G_{1,x}^{(q)}(u_{1},u_{2})\text{d}x\text{d}u_{1}\text{d}u_{2}.
\end{align}
From \eqref{Ggeq0} we deduce that the integrand in the right-hand side of \eqref{eq3n:14/10/2015} is positive and increasing as $\varepsilon$ decreases to zero. Therefore, applying the monotone convergence theorem in relation \eqref{eq3n:14/10/2015} we obtain \eqref{eq5:20/07/2015}. The constant $\sigma_{q}^{2}$ is finite by Lemma \ref{Lema:aux_regionsqg2s}.

To prove relation \eqref{eq10n:20/07/2015}, notice that equations \eqref{eq1a:06/08/2015} and \eqref{eq2a:06/08/2015} imply that  
\begin{align}\label{eq5n:14/10/2015}
\E\left[I_{2q-1}(f_{2q-1,\varepsilon})^2\right]
  &=2(2q-1)!\beta_{q}^{2}\int_{\Sc}G_{\varepsilon,s_{2}-s_{1}}^{(q)}(t_{1}-s_{1},t_{2}-s_{2})\text{d}s_{1}\text{d}s_{2}\text{d}t_{1}\text{d}t_{2}.
\end{align}
Relation \eqref{eq10n:20/07/2015} follows by applying the monotone convergence theorem to \eqref{eq5n:14/10/2015}. To prove that $\overline{\sigma}_{q}$ is finite we change the coordinates $(s_{1},s_{2},t_{1},t_{2})$ by $(s_{1},x:=s_{2}-s_{1},u_{1}:=t_{1}-s_{1},u_{2}:=t_{2}-s_{2})$  in the integral of the right-hand side of \eqref{eq11n:20/07/2015}, to get 
\begin{align*}
\int_{\Sc}G_{0,s_{2}-s_{1}}^{(q)}(t_{1}-s_{1},t_{2}-s_{2})\text{d}s_{1}\text{d}s_{2}\text{d}t_{1}\text{d}t_{2}
  &\leq \int_{[0,T]^{4}}G_{0,x}^{(q)}(u_{1},u_{2})\text{d}s_{1}\text{d}x\text{d}u_{1}\text{d}u_{2}\\
	&= T\int_{[0,T]^{3}}G_{0,x}^{(q)}(u_{1},u_{2})\text{d}x\text{d}u_{1}\text{d}u_{2}.
\end{align*}
The latter integral is finite by Lemma \ref{Lema:aux_regionsqg2s}. Therefore, the constant $\overline{\sigma}_{q}^{2}$ is finite.
\end{proof}
\section{Limit behavior of \texorpdfstring{$\alpha_{\varepsilon}$}{TEXT} and \texorpdfstring{$I_{2q-1}\left(f_{2q-1,\varepsilon}\right)$}{TEXT}}\label{section:limit_theorems}
The next result is a central limit theorem for $\alpha_\varepsilon$ in case $\frac{2}{3}<H<1$.
\begin{Teorema}\label{theorem:l02limit/02/04}
Let $T,\varepsilon>0$ and $\frac{2}{3}<H<1$ be fixed. Then 
\begin{align}\label{eq1:03/05/2015}
\varepsilon^{\frac{3}{2}-\frac{1}{H}}\alpha_{\varepsilon}
  &\stackrel{Law}{\rightarrow}\mathcal{N}(0,\sigma^{2}),~~~~~~\text{when}~~\varepsilon\rightarrow0,
\end{align}
where $\sigma^{2}$ is defined  by \eqref{eq:5:01/04/2015}.
\end{Teorema}
\begin{proof}
Let $f_{2q-1,\varepsilon}$ be defined by \eqref{fdef27/01}. By equation \eqref{chaos:alpha}, 
\begin{align*}
\varepsilon^{\frac{3}{2}-\frac{1}{H}}\alpha_{\varepsilon}
 &=\varepsilon^{\frac{3}{2}-\frac{1}{H}}I_{1}\left(f_{1,\varepsilon}\right) +\varepsilon^{\frac{3}{2}-\frac{1}{H}}\sum_{q=2}^{\infty}I_{2q-1}\left(f_{2q-1,\varepsilon}\right).
\end{align*}
By Lemma \ref{lemma:varianza1:28/04},  the variance of $\varepsilon^{\frac{3}{2}-\frac{1}{H}}I_{1}\left(f_{1,\varepsilon}\right)$ converges to $\sigma^{2}$, where $\sigma^{2}$ is defined  by \eqref{eq:5:01/04/2015}. In addition, combining Lemmas \ref{theorem:variance}  and \ref{lemma:varianza1:28/04}, it follows that	the term 
\begin{align*}
\varepsilon^{\frac{3}{2}-\frac{1}{H}}\sum_{q=2}^{\infty}I_{2q-1}\left(f_{2q-1,\varepsilon}\right)
\end{align*}
converges to zero in $L^{2}$. Then \eqref{eq1:03/05/2015} follows from the fact that $\varepsilon^{\frac{3}{2}-\frac{1}{H}}I_{1}\left(f_{1,\varepsilon}\right)$ is Gaussian and its variance converges to $\sigma^{2}$.
\end{proof}
In the next result we describe the asymptotic behavior of the chaotic components of $\alpha_{\varepsilon}$ in the case $\frac{2}{3}<H<1$.
\begin{Teorema}\label{theorem:l2limit}
Let $T,\varepsilon>0$ and $q\in\N$, $q\geq2$ be fixed. Define $f_{2q-1,\varepsilon}$ by \eqref{fdef27/01}. If $\frac{2}{3}<H<\frac{3}{4}$, then $I_{2q-1}(f_{2q-1,\varepsilon})$ converges in $L^{2}$ when $\varepsilon\rightarrow0$.
\end{Teorema}
\begin{proof}
Define  $f_{2q-1,\varepsilon,s,t}$   by \eqref{eq:kernel_definition}. For every $\varepsilon,\eta>0$ we have 
\begin{align*}
\E\left[\left(I_{2q-1}(f_{2q-1,\varepsilon})-I_{2q-1}(f_{2q-1,\eta})\right)^2\right]
  &= \E\left[I_{2q-1}(f_{2q-1,\varepsilon})^2\right]+\E\left[I_{2q-1}(f_{2q-1,\eta})^2\right]\\
	&\quad -2\E\left[I_{2q-1}(f_{2q-1,\varepsilon})I_{2q-1}(f_{2q-1,\eta})\right].
\end{align*}
 Define $\Rc$ and $\Sc$ by \eqref{fdef27asdn/01} and \eqref{eq:1:regions}, respectively. Then we have 
\begin{multline}\label{eqy7qwxgfqbfbf12/08/2015}
\E\left[I_{2q-1}\left(f_{2q-1,\varepsilon}\right)I_{2q-1}\left(f_{2q-1,\eta}\right)\right]\\
\begin{aligned}
  &= (2q-1)!\Ip{f_{2q-1,\varepsilon},f_{2q-1,\eta}}_{\Hg^{\otimes(2q-1)}}\\
	&= (2q-1)!\Ip{\int_{\Rc}f_{2q-1,\varepsilon,s,t}\text{d}s\text{d}t,\int_{\Rc}f_{2q-1,\eta,s,t}\text{d}s\text{d}t}_{\Hg^{\otimes(2q-1)}}\\
	&= 2(2q-1)!\int_{\Sc}\Ip{f_{2q-1,\varepsilon,s_{1},t_{1}},f_{2q-1,\eta,s_{2},t_{2}}}_{\Hg^{\otimes(2q-1)}}\text{d}s_{1}\text{d}s_{2}\text{d}t_{1}\text{d}t_{2}.
\end{aligned}
\end{multline}
Substituting \eqref{eq2a:06/08/2015} into  \eqref{eqy7qwxgfqbfbf12/08/2015}, yields
\[
\E\left[I_{2q-1}(f_{2q-1,\varepsilon})I_{2q-1}(f_{2q-1,\eta})\right]
= 2(2q-1)!\beta_{q}^2 \int_{\Sc}	G_{\varepsilon,s_{2}-s_{1}}^{(q)}(t_{1}-s_{1},t_{2}-s_{2})\text{d}s_{1}\text{d}s_{2}\text{d}t_{1}\text{d}t_{2},
\]
where $G_{\varepsilon ,x}^{(q)}(u_{1},u_{2})$ is defined by \eqref{Gdef}. Since  $G_{\varepsilon ,x}^{(q)}(u_{1},u_{2})$ is nonnegative (see equation  \eqref{Ggeq0}), the integral in the right-hand side of the previous identity is positive and decreasing in the variables $\varepsilon$, $\eta$. Hence, by the monotone convergence theorem, as $\varepsilon, \eta \rightarrow 0$, the terms
$\E\left[I_{2q-1}(f_{2q-1,\varepsilon})I_{2q-1}(f_{2q-1,\eta})\right]$, $\E\left[I_{2q-1}(f_{2q-1,\varepsilon})^2\right]$ and $\E\left[I_{2q-1}(f_{2q-1,\eta})^2\right]$ converge to
\begin{equation} \label{eq5:28/04/2015}
2(2q-1)!\beta_{q}^2 \int_{\Sc}	G_{0,s_{2}-s_{1}}^{(q)}(t_{1}-s_{1},t_{2}-s_{2})\text{d}s_{1}\text{d}s_{2}\text{d}t_{1}\text{d}t_{2}.
\end{equation}
The previous quantity is finite thanks to Lemma \ref{lemma:varianza:28/04}. From the previous analysis we conclude that the sequence $\{I_{2q-1}(f_{2q-1,\varepsilon_{n}})\}_{n\in\N}$ is Cauchy in $L^{2}$, for any sequence $\{\varepsilon_{n}\}_{n\in\N}\subset[0,1]$ such that $\varepsilon_{n}\rightarrow0$ as $n\rightarrow\infty$, which implies the desired result.
\end{proof}
The next result is a central limit theorem for $I_{2q-1}(f_{2q-1,\varepsilon})$ in the case $\frac{3}{2}<H<\frac{4q-3}{4q-2}$.
\begin{Teorema}\label{theorem:distributionchaos}
Let $T,\varepsilon>0$ and $q\in\N$, $q\geq2$ be fixed. Define $f_{2q-1,\varepsilon}$ by \eqref{fdef27/01}. Then, for every  $\frac{3}{4}<H<\frac{4q-3}{4q-2}$ we have
\begin{align}\label{eq1:08/03/2015}
\varepsilon^{1-\frac{3}{4H}}I_{2q-1}(f_{2q-1,\varepsilon})\stackrel{Law}{\rightarrow} \mathcal{N}(0,\sigma_{q}^2),~~~~~~\text{when}~~\varepsilon\rightarrow0,
\end{align}
where $\sigma_{q}^2$ is the finite constant defined by \eqref{eq6:20/07/2015}.
\end{Teorema}
\begin{proof} 
Define $f_{2q-1,\varepsilon,s,t}$, for $0\leq s\leq t$,  by \eqref{eq:kernel_definition} and $\Rc$ by 	\eqref{fdef27asdn/01}. By \eqref{fdef27/01},
\begin{align*}
\varepsilon^{1-\frac{3}{4H}}I_{2q-1}(f_{2q-1,\varepsilon})
  &= (-1)^{q}\varepsilon^{1-\frac{3}{4H}}\int_{\Rc}\beta_{q}(\varepsilon+(t-s)^{2H})^{-\frac{1}{2}-q}I_{2q-1}\left(\Indi{[s,t]}^{\otimes(2q-1)}\right)\text{d}s\text{d}t.
\end{align*}
Then, using the self-similarity of the fractional Brownian motion we get
\begin{multline*}
\varepsilon^{1-\frac{3}{4H}}I_{2q-1}(f_{2q-1,\varepsilon})\\
\begin{aligned}
&\stackrel{Law}{=} (-1)^{q}\varepsilon^{1-\frac{3}{4H}}\int_{\Rc}\beta_{q}(\varepsilon+(t-s)^{2H})^{-\frac{1}{2}-q}I_{2q-1}\left(\left(\sqrt{\varepsilon}\Indi{\varepsilon^{-\frac{1}{2H}}[s,t]}\right)^{\otimes(2q-1)}\right)\text{d}s\text{d}t.
\end{aligned}
\end{multline*}
Therefore, changing the coordinates $(s,t)$ by $(\varepsilon^{-\frac{1}{2H}}s,\varepsilon^{-\frac{1}{2H}}t)$ we get 
\begin{multline}\label{eq1a:14/07/2015}
\varepsilon^{1-\frac{3}{4H}}I_{2q-1}(f_{2q-1,\varepsilon})\\
\begin{aligned}
	&\stackrel{Law}{=} (-1)^{q}\varepsilon^{\frac{1}{4H}}\int_{\varepsilon^{-\frac{1}{2H}}\Rc}\beta_{q}\left(1+(t-s)^{2H}\right)^{-\frac{1}{2}-q}I_{2q-1}\left(\Indi{[s,t]}^{\otimes(2q-1)}\right)\text{d}s\text{d}t\\
	&= \varepsilon^{\frac{1}{4H}}\int_{\varepsilon^{-\frac{1}{2H}}\Rc}I_{2q-1}\left(f_{2q-1,1,s,t}\right)\text{d}s\text{d}t.
\end{aligned}
\end{multline}
Changing the coordinates $(s,t)$ by $(s,u:=t-s)$ in \eqref{eq1a:14/07/2015}, and defining $N:=\varepsilon^{-\frac{1}{2H}}$, we obtain 
\begin{align}\label{eq1m:17/07/2015}
\varepsilon^{1-\frac{3}{4H}}I_{2q-1}(f_{2q-1,\varepsilon})
&\stackrel{Law}{=} \frac{1}{\sqrt{N}}\int_{0}^{NT}\int_{0}^{NT-s} I_{2q-1}\left(f_{2q-1,1,s,s+u}\right)\text{d}u\text{d}s.
\end{align}
From \eqref{eq1m:17/07/2015} it follows that the convergence \eqref{eq1:08/03/2015} is equivalent to 
\begin{align}\label{eq2m:17/07/2015}
\frac{1}{\sqrt{N}}\int_{0}^{NT}\int_{0}^{NT-s} I_{2q-1}\left(f_{2q-1,1,s,s+u}\right)\text{d}u\text{d}s
 &\stackrel{Law}{\rightarrow}\mathcal{N}(0,\sigma_{q}^{2}),~~~~~~\text{ as }~~N\rightarrow\infty.
\end{align}
The proof of \eqref{eq2m:17/07/2015} will be done in several steps. \\~\\
\textit{Step I}\\ 
Define the random variable
\begin{align*}
Y_{N}
	&:=\frac{1}{\sqrt{N}}\int_{0}^{NT}\int_{NT-s}^{\infty}I_{2q-1}\left(f_{2q-1,1,s,s+u}\right)\text{d}u\text{d}s.
\end{align*}
First we show that $Y_{N}$ converges to zero in $L^{2}$ as $N\rightarrow\infty$. Notice that 
\begin{align}
\E\left[Y_{N}^2\right]
  &=\frac{2}{N}\int_{0}^{NT}\int_{0}^{NT}\int_{NT-s_{2}}^{\infty}\int_{NT-s_{1}}^{\infty}\Ind{s_{1}\leq s_{2}}\nonumber\\
	&\times\E\left[I_{2q-1}\left(f_{2q-1,1,s_{1},s_{1}+u_{1}}\right)I_{2q-1}\left(f_{2q-1,1,s_{2},s_{2}+u_{2}}\right)\right]\text{d}u_{1}\text{d}u_{2}\text{d}s_{1}\text{d}s_{2}\nonumber\\
	&= \frac{2(2q-1)!}{N}\int_{0}^{NT}\int_{0}^{NT}\int_{NT-s_{2}}^{\infty}\int_{NT-s_{1}}^{\infty}\Ind{s_{1}\leq s_{2}}\nonumber\\
	&\times\Ip{f_{2q-1,1,s_{1},s_{1}+u_{1}},f_{2q-1,1,s_{2},s_{2}+u_{2}}}_{\Hg^{\otimes(2q-1)}}\text{d}u_{1}\text{d}u_{2}\text{d}s_{1}\text{d}s_{2}.\label{eq0sss.1:03/05/2015}
\end{align}
Define the function $G_{1,x}^{(q)}(v_{1},v_{2})$, $x,v_{1},v_{2}\geq0$, as in \eqref{Gdef}. Substituting equation \eqref{eq2a:06/08/2015} in \eqref{eq0sss.1:03/05/2015}, and changing the order of integration, we get
\begin{align}
\E\left[Y_{N}^2\right]
	&= \frac{2(2q-1)!\beta_{q}^{2}}{N}\int_{0}^{\infty}\int_{0}^{\infty}\int_{0\vee(NT-u_{2})}^{NT}\int_{0\vee(NT-u_{1})}^{NT}
	\Ind{s_{1}\leq s_{2}}\nonumber\\
	&\times G_{1,s_{2}-s_{1}}^{(q)}(u_{1},u_{2})\text{d}s_{1}\text{d}s_{2}\text{d}u_{1}\text{d}u_{2}.\label{eq0ns.1:03/05/2015}
\end{align}
Changing the coordinates $(s_{1},s_{2},u_{1},u_{2})$ by $(s_{1},x:=s_{2}-s_{1},u_{1},u_{2})$ in the right hand side of \eqref{eq0ns.1:03/05/2015}, we get 
\begin{align*}
\E\left[Y_{N}^2\right]
	&\leq \frac{2(2q-1)!\beta_{q}^{2}}{N}\int_{\R_{+}^{3}}\int_{0\vee (NT-u_{1})}^{NT}G_{1,x}^{(q)}(u_{1},u_{2})\text{d}s_{1}\text{d}x\text{d}u_{1}\text{d}u_{2},
\end{align*}
and then integrating the $s_{1}$ variable, 
\begin{align}
\E\left[Y_{N}^2\right]
	&\leq 2(2q-1)!\beta_{q}^{2}\int_{\R_{+}^{3}}\left(T-\frac{0\vee(NT-u_{1})}{N}\right)G_{1,x}^{(q)}(u_{1},u_{2})\text{d}x\text{d}u_{1}\text{d}u_{2}.\label{eq0.1:03/05/2015}
\end{align}
The integrand in \eqref{eq0.1:03/05/2015} converges to zero pointwise, and is dominated by the function
\begin{align*}
2(2q-1)!\beta_{q}^{2}TG_{1,x}^{(q)}(u_{1},u_{2}).
\end{align*}
By condition $H<\frac{4q-3}{4q-2}$ and Lemma 5.8,  the function  $G_{1,x}^{(q)}(u_{1},u_{2})$ is integrable in $\R^3_+$.
Hence, applying the dominated convergence theorem to \eqref{eq0.1:03/05/2015}, we obtain $\E\left[Y_{N}^2\right]\rightarrow0$, as $N\rightarrow\infty$ as required.\\\\
\textit{Step II}\\
Since $Y_{N}\rightarrow0$ in $L^{2}$ as $N\rightarrow\infty$, to prove the convergence \eqref{eq2m:17/07/2015} it suffices to show that the random variable 
\begin{align*}
J_{2q-1,N}
  &:= \frac{1}{\sqrt{N}}\int_{0}^{N T}\int_{0}^{\infty} I_{2q-1}\left(f_{2q-1,1,s,s+u}\right)\text{d}u\text{d}s,
\end{align*}
converges in law to a Gaussian distribution with variance $\sigma_{q}^{2}$ as $N\rightarrow\infty$. For $M\in\N$, $M\geq1$ fixed, consider the following Riemann sum approximation for $J_{2q-1,N}$
\begin{align*}
\widetilde{J}_{2q-1,M,N}
 &:=\frac{1}{2^{M}}\sum_{k=2}^{M2^{M}}\frac{1}{\sqrt{N}}\int_{0}^{NT}I_{2q-1}\left(f_{2q-1,1,s,s+u(k)}\right)\text{d}s,
\end{align*}
where $u(k):=\frac{k}{2^{M}}$, for $k=2,\dots,M2^{M}$. We will prove that $\widetilde{J}_{2q-1,M,N}\rightarrow J_{2q-1,N}$ in $L^{2}$ as $M\rightarrow\infty$ uniformly in $N>1$, and $\widetilde{J}_{2q-1,M,N}\rightarrow \mathcal{N}(0,\widetilde{\sigma}_{q,M}^{2})$ as $N\rightarrow\infty$  for some constant $\widetilde{\sigma}_{q,M}^{2}$ satisfying $\widetilde{\sigma}_{q,M}^{2}\rightarrow \sigma_{q}^{2}$ as $M\rightarrow\infty$. The result will then follow by a standard approximation argument. We will separate the argument in the following steps.\\\\ 
\textit{Step III}\\
Next we prove that prove that $\widetilde{J}_{2q-1,M,N}\rightarrow J_{2q-1,N}$ in $L^{2}$ as $M\rightarrow\infty$ uniformly in $N>1$, namely,
\begin{align}\label{eq9:28/04/2015}
\lim_{M\rightarrow\infty}\sup_{N>1}\Norm{J_{2q-1,N}-\widetilde{J}_{2q-1,M,N}}_{L^2}=0.
\end{align}
For $M\in\N$ fixed, we decompose  the term $J_{2q-1,N}$ as 
\begin{align}	
J_{2q-1,N}
  &= J_{2q-1,M,N}^{(1)}+J_{2q-1,M,N}^{(2)},\label{eq2n:13/08/2015}
\end{align}
where 
\begin{align*}
J_{2q-1,M,N}^{(1)}
  &:=\frac{1}{\sqrt{N}}\int_{0}^{N T}\int_{2^{-M}}^{M} I_{2q-1}\left(f_{2q-1,1,s,s+u}\right)\text{d}u\text{d}s
\end{align*}
and 
\begin{align*}
J_{2q-1,M,N}^{(2)}
  &:=\frac{1}{\sqrt{N}}\int_{0}^{N T}\int_{0}^{\infty}\Indi{(0,2^{-M})\cup(M,\infty)}(u) I_{2q-1}\left(f_{2q-1,1,s,s+u}\right)\text{d}u\text{d}s.
\end{align*}
From \eqref{eq2n:13/08/2015} we deduce that relation \eqref{eq9:28/04/2015} is equivalent to 
\begin{align}
\lim_{M\rightarrow\infty}\sup_{N>1}\Norm{J_{2q-1,M,N}^{(1)}-\widetilde{J}_{2q-1,M,N}}_{L^{2}}
  &=0\label{eq1m:13/08/2015},
\end{align}
provided that 
\begin{align}
\lim_{M\rightarrow\infty}\sup_{N>1}\Norm{J_{2q-1,M,N}^{(2)}}_{L^{2}}
  &=0\label{eq1n:13/08/2015}.
\end{align}
To prove \eqref{eq1n:13/08/2015} we proceed as follows. First we write
\begin{align}\label{eq510n:14/08/2015}
\Norm{J_{2q-1,M,N}^{(2)}}_{L^{2}}^{2}
  &=\frac{2(2q-1)!}{N}\int_{\R_{+}^{2}}\int_{[0,NT]^2}\Indi{(0,2^{-M})\cup(M,\infty)}(u_{1})\Indi{(0,2^{-M})\cup(M,\infty)}(u_{2})\nonumber\\
  &~~\times \Indi{\{s_{1}\leq s_{2}\}}\Ip{f_{2q-1,1,s_{1},s_{1}+u_{1}},f_{2q-1,1,s_{2},s_{2}+u_{2}}}_{\Hg^{\otimes(2q-1)}}\text{d}s_{1}\text{d}s_{2}\text{d}u_{1}\text{d}u_{2}.
\end{align}
Let $G_{1,x}^{(q)}(v_{1},v_{2})$, $x,v_{1},v_{2}\in\R_{+}$ be defined by \eqref{Gdef}.  Applying identity \eqref{eq2a:06/08/2015} in \eqref{eq510n:14/08/2015}, and then changing the coordinates $(s_{1},s_{2},u_{1},u_{2})$ by $(s_{1},x:=s_{2}-s_{1},u_{1},u_{2})$  in \eqref{eq510n:14/08/2015}, we get 
\begin{align}\label{eq4sdfff:20/04/2015}
\Norm{J_{2q-1,M,N}^{(2)}}_{L^{2}}^{2}
  &\leq \frac{2(2q-1)!\beta_{q}^{2}}{N}\int_{\R_{+}^{3}}\int_{0}^{NT}
\Indi{(0,2^{-M})\cup(M,\infty)}(u_{1})\nonumber\\
&~\times\Indi{(0,2^{-M})\cup(M,\infty)}(u_{2})G_{1,x}^{(q)}(u_{1},u_{2})\text{d}s_{1}\text{d}x\text{d}u_{1}\text{d}u_{2}.
\end{align}
Integrating the variable $s_{1}$ in \eqref{eq4sdfff:20/04/2015} we obtain
\begin{align}\label{eq:31}
\Norm{J_{2q-1,M,N}^{(2)}}_{L^{2}}^{2}
  &\leq 2T(2q-1)!\beta_{q}^{2}\int_{\R_{+}^{3}}\Indi{(0,2^{-M})\cup(M,\infty)}(u_{2})\nonumber\\
	&~\times\Indi{(0,2^{-M})\cup(M,\infty)}(u_{2})G_{1,x}^{(q)}(u_{1},u_{2})\text{d}x\text{d}u_{1}\text{d}u_{2}.
\end{align}
The integrand is dominated by the function $2(2q-1)!\beta_{q}^{2}TG_{1,x}^{(q)}(u_{1},u_{2})$, which is integrable by the condition $H<\frac{2q-3}{4q-2}$, and Lemma \ref{Lema:aux_regionlqg2}. Hence, applying the dominated convergence theorem  to \eqref{eq:31}, we get \eqref{eq1n:13/08/2015}.

To  prove \eqref{eq1m:13/08/2015} we proceed as follows. For $k=2,\dots,M2^{M}$ define the interval $I_{k}:=\left(\frac{k-1}{2^{M}},\frac{k}{2^{M}}\right]$. Notice that $J_{2q-1,M,N}^{(1)}$ and $\widetilde{J}_{2q-1,M,N}$ can be written, respectively,  as 
\begin{equation}
J_{2q-1,M,N}^{(1)}
  =\frac{1}{\sqrt{N}}\int_{0}^{NT}\int_{\R_{+}}\sum_{k=2}^{M2^{M}}I_{2q-1}\left(f_{2q-1,1,s,s+u}\right)\Indi{I_{k}}(u)\text{d}u\text{d}s,\label{eq21p:10/07/2015}
  \end{equation}
  and
  \begin{equation}
\widetilde{J}_{2q-1,M,N}
  =\frac{1}{\sqrt{N}}\int_{0}^{NT}\int_{\R_{+}}\sum_{k=2}^{M2^{M}}I_{2q-1}\left(f_{2q-1,1,s,s+u(k)}\right)\Indi{I_{k}}(u)\text{d}u\text{d}s.\label{eq21:10/07/2015}
\end{equation}
Applying \eqref{eq2a:06/08/2015}, we can prove that 
\begin{align}
\Norm{J_{2q-1,M,N}^{(1)}-\widetilde{J}_{2q-1,M,N}}_{L^{2}}^{2}
  &=\frac{2(2q-1)!\beta_{q}^{2}}{N}\int_{\R_{+}^{2}}\int_{[0,NT]^{2}}\sum_{k_{1},k_{2}=2}^{M2^{M}}\Indi{I_{k_{1}}}(u_{1})\Indi{I_{k_{2}}}(u_{2})\nonumber\\
	&\times \Indi{\{s_{1}\leq s_{2}\}}\Theta_{k_{1},k_{2}}^{(q)}(s_{2}-s_{1},u_{1},u_{2})\text{d}s_{1}\text{d}s_{2}\text{d}u_{1}\text{d}u_{2},\label{eq6m:13/08/2015a}
\end{align}
where the function $\Theta_{k_{1},k_{2}}^{(q)}$ is defined by
\begin{align*}
\Theta_{k_{1},k_{2}}^{(q)}(x,u_{1},u_{2})
  &:=\bigg(G_{1,x}^{(q)}(u_{1},u_{2})-G_{1,x}^{(q)}(u(k_{1}),u_{2})\\
	&~~-G_{1,x}^{(q)}(u_{1},u(k_{2}))+G_{1,x}^{(q)}(u(k_{1}),u(k_{2}))\bigg).
\end{align*}
Changing the coordinates $(s_{1},s_{2},u_{1},u_{2})$ by $(s_{1},x:=s_{2}-s_{1},u_{1},u_{2})$, and then integrating the $s_{1}$ variable in \eqref{eq6m:13/08/2015a}, we obtain
\begin{align*}
\Norm{J_{2q-1,M,N}^{(1)}-\widetilde{J}_{2q-1,M,N}}_{L^{2}}^{2}
  &=2(2q-1)!\beta_{q}^{2}\int_{\R_{+}^{2}}\int_{0}^{NT}\sum_{k_{1},k_{2}=2}^{M2^{M}}\Indi{I_{k_{1}}}(u_{1})\Indi{I_{k_{2}}}(u_{2})\nonumber\\
	&~~\times \left(T-\frac{x}{N}\right)\Theta_{k_{1},k_{2}}^{(q)}(x,u_{1},u_{2})\text{d}x\text{d}u_{1}\text{d}u_{2}.
\end{align*}
As a consequence, 
\begin{align*}
\Norm{J_{2q-1,M,N}^{(1)}-\widetilde{J}_{2q-1,M,N}}_{L^{2}}^{2}
  &\leq 2(2q-1)!\beta_{q}^{2}T\int_{\R_{+}^{3}}\sum_{k_{1},k_{2}=2}^{M2^{M}}\Indi{I_{k_{1}}}(u_{1})\Indi{I_{k_{2}}}(u_{2})\nonumber\\
	&~~\times \Theta_{k_{1},k_{2}}^{(q)}(x,u_{1},u_{2})\text{d}x\text{d}u_{1}\text{d}u_{2}.
\end{align*}
By the continuity of $G_{1,x}(u_{1},u_{2})$, the term 
\begin{align*}
\sum_{k_{1},k_{2}=2}^{M2^{M}}\Indi{I_{k_{1}}}(u_{1})\Indi{I_{k_{2}}}(u_{2})\Theta_{k_{1},k_{2}}^{(q)}(x,u_{1},u_{2})
\end{align*}
converges to zero as $M\rightarrow\infty$. Next we prove that this term is dominated by an integrable function. Let $u_{1}\in I_{k_{1}},u_{2}\in I_{k_{2}}$ be fixed.  Notice that $u_{i},u(k_{i})\leq u_{i}+2^{-M}\leq u_{i}+1$ for $i=1,2$. Hence, applying Lemma \ref{remarkG}, we deduce that the terms $G_{1,x}^{(q)}(u_{1},u_{2})$, $G_{1,x}^{(q)}(u(k_{1}),u_{2})$, $G_{1,x}^{(q)}(u_{1},u(k_{2}))$ and  $G_{1,x}^{(q)}(u(k_{1}),u(k_{2}))$ are bounded by $KG_{1,x}^{(q)}(u_{1}+1,u_{2}+1)$, for some constant $K>0$ only depending on $H$ and $q$. As a consequence, 
\begin{align*}
\sum_{k_{1},k_{2}=2}^{M2^{M}}\Indi{I_{k_{1}}}(u_{1})\Indi{I_{k_{2}}}(u_{2})\Theta_{k_{1},k_{2}}^{(q)}(x,u_{1},u_{2})
  &\leq 4K  G_{1,x}^{(q)}(u_{1}+1,u_{2}+1),
\end{align*}
for some constant $K$ only depending on $H$ and $q$.  Therefore,   
 the right-hand side of the previous identity is integrable over $x,u_{1},u_{2}>0$ due to  Lemma \ref{Lema:aux_regionlqg2}, since 
\begin{align}\label{eq3n:01/11/2015}
 \int_{\R_{+}^{3}}G_{1,x}^{(q)}(u_{1}+1,u_{2}+1)\text{d}x\text{d}u_{1}\text{d}u_{2}
 &=\int_{[1,\infty)^{2}}\int_{\R_{+}}G_{1,x}^{(q)}(u_{1},u_{2})\text{d}x\text{d}u_{1}\text{d}u_{2}\nonumber\\
 &\leq\int_{\R_{+}^{3}} G_{1,x}^{(q)}(u_{1},u_{2})\text{d}x\text{d}u_{1}\text{d}u_{2}<\infty.
\end{align}
This finishes the proof of \eqref{eq1m:13/08/2015}.\\\\
\textit{Step IV}\\
Next we  prove that
\begin{align}\label{eq3a:16/07/2015}
\lim_{N\rightarrow\infty}\E\left[\widetilde{J}_{2q-1,M,N}^2\right]
  &=\widetilde{\sigma}_{q,M}^{2}, 
\end{align}	
where $\widetilde{\sigma}_{q,M}^{2}$ is the finite constant defined by 
\begin{align}
\widetilde{\sigma}_{q,M}^{2}
  &:=(2q-1)!\beta_{q}^{2}2^{1-2M}T\sum_{k_{1},k_{2}=2}^{M2^{M}}\int_{0}^{\infty}G_{1,x}^{(q)}(u(k_{1}),u(k_{2}))\text{d}x\label{eq3m:17/07/2015}.
\end{align}
In addition, we will prove that $\widetilde{\sigma}_{q,M}^{2}$ satisfies
\begin{align}
\lim_{M\rightarrow\infty}\widetilde{\sigma}_{q,M}^{2}
  &=\sigma_{q}^{2}\label{eq1n:14/07/2015},
\end{align}
where $\sigma_{q}^{2}$ is defined by \eqref{eq6:20/07/2015}. In order to prove \eqref{eq3a:16/07/2015} and \eqref{eq1n:14/07/2015} we proceed as follows. From \eqref{eq21:10/07/2015}, we can prove that
\begin{align*}
\E \left[\widetilde{J}_{2q-1,M,N}^{2}\right]
 &=\int_{\R_{+}^{3}}Q_{M,N}(x,u_{1},u_{2})\text{d}x\text{d}u_{1}\text{d}u_{2},
\end{align*}
where 
\begin{align*}
Q_{M,N}(x,u_{1},u_{2})
  &:=2(2q-1)!\Indi{[0,NT]}(x)\beta_{q}^{2}\sum_{k_{1},k_{2}=2}^{M2^{M}}\left(T-\frac{x}{N}\right)\nonumber\\
 &~\times G_{1,x}^{(q)}(u(k_{1}),u(k_{2}))\Indi{I_{k_{1}}}(u_{1})\Indi{I_{k_{2}}}(u_{2}).
\end{align*}
Notice that $Q_{M,N}$ satisfies
\begin{align}
\lim_{N\rightarrow\infty}Q_{M,N}(x,u_{1},u_{2})
  &=Q_{M}(x,u_{1},u_{2}),\label{eq4m:15/07/2015}
\end{align}
where $Q_{M}$ is defined by
\begin{align*}
Q_{M}(x,u_{1},u_{2})
  &:=2(2q-1)!\beta_{q}^{2}T\sum_{k_{1},k_{2}=2}^{M2^{M}}G_{1,x}^{(q)}(u(k_{1}),u(k_{2}))\Indi{I_{k_{1}}}(u_{1})\Indi{I_{k_{2}}}(u_{2}).
\end{align*}
In turn, $Q_{M}$ satisfies 
\begin{align}
\lim_{M\rightarrow\infty}Q_{M}(x,u_{1},u_{2})
  &=Q(x,u_{1},u_{2}),\label{eq5m:15/07/2015}
\end{align}
where $Q$ is defined by 
\begin{align*}
Q(x,u_{1},u_{2})
  &:=2(2q-1)!\beta_{q}^{2}TG_{1,x}^{(q)}(u_{1},u_{2}).
\end{align*}
Let $x>0$ and $2\leq k_{1},k_{2}\leq M2^{M}$ be fixed, and take $u_{i}\in I_{k_{i}}$, $i=1,2$. Since $u(k_{i})\leq u_{i}+2^{-M}\leq u_{i}+1$, by Lemma \ref{remarkG},   there exists a constant $K>0$, only depending on $q$ and $H$, such that
\begin{align*}
G_{1,x}^{(q)}(u(k_{1}),u(k_{2}))
  &\leq KG_{1,x}^{(q)}(u_{1}+1,u_{2}+1),
\end{align*}
  As a consequence, there exists a constant $K$ only depending on $q,H$ and $T$ such that
\begin{align}
Q_{M,N}(x,u_{1},u_{2})
  &\leq KG_{1,x}^{(q)}(u_{1}+1,u_{2}+1)\label{eq6m:15/07/2015},
\end{align}
and, hence, 
\begin{align}
Q_{M}(x,u_{1},u_{2})
  &\leq KG_{1,x}^{(q)}(u_{1}+1,u_{2}+1)\label{eq6m:15/07/2015ss}.
\end{align}
The function $G_{1,x}^{(q)}(u_{1}+1,u_{2}+1)$ is integrable with respect to the variables $x,u_{1},u_{2}>0$ thanks to \eqref{eq3n:01/11/2015}. Hence, taking into account  \eqref{eq4m:15/07/2015} and  \eqref{eq5m:15/07/2015}, as well as the estimates \eqref{eq6m:15/07/2015} and 
\eqref{eq6m:15/07/2015ss}, we can apply the dominated convergence theorem twice, to obtain
\begin{align}
\lim_{M\rightarrow\infty}\lim_{N\rightarrow\infty}\E \left[\widetilde{J}_{2q-1,M,N}^{2}\right]
&=\lim_{M\rightarrow\infty}\int_{\R_{+}^{3}}Q_{M}(x,u_{1},u_{2})\text{d}x\text{d}u_{1}\text{d}u_{2}\nonumber\\
&=\int_{\R_{+}^{3}}Q(x,u_{1},u_{2})\text{d}x\text{d}u_{1}\text{d}u_{2}.\label{eq30cfmasdaaddasdaf:15/07/2015}
\end{align}
Equations \eqref{eq3a:16/07/2015} and \eqref{eq1n:14/07/2015} then follow from   \eqref{eq30cfmasdaaddasdaf:15/07/2015}.\\\\
\textit{Step V}\\
Next we prove the convergence in law of $J_{2q-1,N}$ to a Gaussian random variable with variance $\sigma_{q}^{2}$, which we will denote by $\mathcal{N}(0,\sigma_{q}^{2})$. Let $y\in\R$ be fixed. Notice that 
\begin{align}
\Abs{\Pb[J_{2q-1,N}\leq y]-\Pb[\mathcal{N}\left(0,\sigma_{q}^{2}\right)\leq y]}
  &\leq \sup_{N>1}\Abs{\Pb\left[J_{2q-1,N}\leq y\right]-\Pb\left[\widetilde{J}_{2q-1,M,N}\leq y\right]}\nonumber\\
	&+ \Abs{\Pb\left[\widetilde{J}_{2q-1,M,N}\leq y\right]-\Pb\left[\mathcal{N}(0,\widetilde{\sigma}_{q,M}^{2})\leq y\right]}\nonumber\\
	&+\Abs{\Pb\left[\mathcal{N}(0,\widetilde{\sigma}_{q,M}^{2})\leq y\right]-\Pb\left[\mathcal{N}(0,\sigma_{q}^{2})\leq y\right]}\label{eq1m:16/07/2015}.
\end{align}
Therefore, if we prove that for $M>0$ fixed
\begin{align}\label{eq10:30/04/20/15}
\widetilde{J}_{2q-1,M,N}\stackrel{Law}{\rightarrow} \mathcal{N}\left(0,\widetilde{\sigma}_{q,M}^{2}\right)~~~~~~\text{ as }~~N\rightarrow\infty,
\end{align}
then from \eqref{eq1m:16/07/2015} we get
\begin{align}
\limsup_{N\rightarrow\infty}\Abs{\Pb[J_{2q-1,N}\leq y]-\Pb[\mathcal{N}(0,\sigma_{q}^{2})\leq y]}
  &\leq \sup_{N>1}\Abs{\Pb\left[J_{2q-1,N}\leq y\right]-\Pb\left[\widetilde{J}_{2q-1,M,N}\leq y\right]}\nonumber\\
	&\quad+\Abs{\Pb\left[\mathcal{N}\left(0,\widetilde{\sigma}_{q,M}^{2}\right)\leq y\right]-\Pb\left[\mathcal{N}\left(0,\sigma_{q}^{2}\right)\leq y\right]}\label{eq2m:16/07/2015},
\end{align}
and hence, from relations \eqref{eq9:28/04/2015}, \eqref{eq1n:14/07/2015} and \eqref{eq2m:16/07/2015}, we conclude that 
\begin{align}\label{eq5m:16/07/2015}
\limsup_{N\rightarrow\infty}\Abs{\Pb[J_{2q-1,N}^2\leq y]-\Pb[\mathcal{N}(0,\sigma_{q}^{2})\leq y]}
  &=0,
\end{align}
and the proof will then be complete. Therefore, it suffices to show \eqref{eq10:30/04/20/15} for $M$ fixed. To prove this first we show that the random vector 
\begin{align*}
Z^{(N)}
  &=\left(Z_{k}^{(N)}\right)_{k=2}^{M2^{M}}:=\left(\frac{1}{\sqrt{N}}\int_{0}^{NT}I_{2q-1}\left(f_{2q-1,1,s,s+u(k)}\right)\text{d}s\right)_{k=2}^{M2^{M}}
\end{align*}
converges to a multivariate Gaussian distribution. By the Peccati-Tudor criterion (see \cite{PTGLVVMSI}), it suffices to prove that the components of the vector $Z^{(N)}$ converge to a Gaussian distribution, and the covariance matrix of $Z^{(N)}$ is convergent.
 
In order to prove that the covariance matrix of $Z^{(N)}$ is convergent we proceed as follows. First, for $2\leq j,k\leq M2^{M}$, we write  
\begin{align*}
\E\left[Z_{k}^{(N)}Z_{j}^{(N)}\right]
  &=\frac{1}{N}\int_{[0,NT]^{2}}\E\left[I_{2q-1}\left(f_{2q-1,1,s_{1},s_{1}+u(k)}\right)I_{2q-1}\left(f_{2q-1,1,s_{2},s_{2}+u(j)}\right)\right]\text{d}s_1\text{d}s_2.
\end{align*}
Then, using \eqref{eq2a:06/08/2015} we get  
\begin{align}
\E\left[Z_{k}^{(N)}Z_{j}^{(N)}\right]
  &=\frac{(2q-1)!\beta_{q}^2}{N}\int_{[0,NT]^2}G_{1,s_{2}-s_{1}}^{(q)}(u(k),u(j))\text{d}s_{1}\text{d}s_{2},\label{eq1a:16/07/2015}
\end{align}
where in the last equality we used the notation $G_{1,-y}(v_{1},v_{2}):=G_{1,y}(v_{2},v_{1})$, for $y,v_{1},v_{2}>0$. Changing the coordinates $(s_{1},s_{2})$ by $(s_{1},x:=s_{2}-s_{1})$ in relation \eqref{eq1a:16/07/2015} and integrating the $s_{1}$, yields
\begin{align}
\E\left[Z_{k}^{(N)}Z_{j}^{(N)}\right]
  &=(2q-1)!\beta_{q}^2\int_{-NT}^{NT}\left(T-\frac{|x|}{N}\right)G_{1,x}^{(q)}(u(k),u(j))\text{d}x.\label{eq2a:16/07/2015}
\end{align}
Finally, applying the monotone convergence theorem in \eqref{eq2a:16/07/2015}, we get 
\begin{align*}
\lim_{N\rightarrow\infty}\E\left[Z_{k}^{(N)}Z_{j}^{(N)}\right]
  &=(2q-1)!\beta_{q}^2T\int_{\R}G_{1,x}^{(q)}(u(k),u(j))\text{d}x,
\end{align*}
which is clearly finite. Thus, we have proved that  the covariance matrix of $Z^{(N)}$ converges to the matrix $\Sigma=(\Sigma_{k,j})_{2\leq k,j\leq M2^{M}}$, where
\begin{align*}
\Sigma_{k,j}
  &:=T(2q-1)!\beta_{q}^2\int_{\R}G_{1,x}^{(q)}(u(k),u(j))\text{d}x.
\end{align*}
Next, for $2\leq k\leq M2^{M}$ fixed, we prove the convergence of $Z_{k}^{(N)}$ to a Gaussian law. By \eqref{eq:kernel_definition},
\begin{align*}
Z_{k}^{(N)}
  &=  \frac {C_{q,k}}{ \sqrt{N}}  \int_{0}^{NT}I_{2q-1}\left(\Indi{[s,s+u_{k}]}^{\otimes(2q-1)}\right)\text{d}s,
\end{align*}
where $C_{q,k}= (-1)^{q}\beta_{q}(1+u_{k}^{2H})^{-\frac{1}{2}-q}$.
Hence, by the self-similarity of the fractional Brownian motion we can write
\begin{align}\label{eq3n:23/11/2015}
Z_{k}^{(N)}
  &\stackrel{Law}{=} \frac{ C_{q,k} }{\sqrt{N}}\int_{0}^{NT}I_{2q-1}\left(\left(u_{k}^{H}N^{H}\Indi{[\frac{s}{Nu_{k}},\frac{s}{Nu_{k}}+\frac{1}{N}]}\right)^{\otimes(2q-1)}\right)\text{d}s.
\end{align}
Making the change of variables $r:=\frac{s}{Nu_{k}}$ in the right hand side of \eqref{eq3n:23/11/2015}, we get 
\begin{align}
Z_{k}^{(N)}
  &\stackrel{Law}{=}  C_{q,k} u_k^{H(2q-1)+1}\sqrt{N}\int_{0}^{\frac{T}{u_{k}}}I_{2q-1}\left(\left(N^{H}\Indi{[r,r+\frac{1}{N}]}\right)^{\otimes(2q-1)}\right)\text{d}r \nonumber\\
   &= C_{q,k} u_k^{H(2q-1)+1}\sqrt{N}\int_{0}^{\frac{T}{u_{k}}} H_{2q-1} \left(  N^H (  B_{r+ \frac{1}{N}} -B_r) \right)\text{d}r. \label{e1}
  \end{align}
where $H_{2q-1}$ denotes the Hermite polynomial of degree $2q-1$. The convergence in law of the right-hand side of \eqref{e1} to a centered Gaussian distribution as $N \rightarrow \infty$ is proven in \cite{NNDLTVolt}, equation (1.3). As a consequence, the components of $Z^{(N)}$ converge to a Gaussian random variable as $N\rightarrow\infty$. Therefore, by the Peccati-Tudor criterion, $Z^{(N)}$ converges in law to a centered Gaussian distribution with covariance $\Sigma$. Hence, 
\begin{align}\label{eq25a:16/07/2015}
\widetilde{J}_{2q-1,M,N}
	&=\frac{1}{2^{2M}}\sum_{k=2}^{M2^{M}}Z_{k}^{(N)}
  \stackrel{Law}{\rightarrow} \mathcal{N}\left(0,\frac{1}{2^{2M}}\sum_{j,k=2}^{M2^{M}}\Sigma_{k,j}\right)\ ~~~~~~ \text{ as }~~N\rightarrow\infty.
\end{align}
The convergence \eqref{eq10:30/04/20/15} follows from \eqref{eq25a:16/07/2015} by using the fact that
\begin{align*}
 \frac{1}{2^{2M}}\sum_{k,j=2}^{M2^{M}}\Sigma_{k,j}
  &=T(2q-1)!\beta_{q}^22^{-2M}\sum_{j,k=2}^{M2^{M}}\int_{\R}G_{1,x}^{(q)}(u(k),u(j))\text{d}x=\widetilde{\sigma}_{q,M}.
\end{align*}
The proof is now complete. 
\end{proof}
\section{Technical lemmas}
In this section we prove several technical results that were used to determine the asymptotic behavior of the variance of $I_{2q-1}\left(f_{2q-1,\varepsilon}\right)$ and $\alpha_{\varepsilon}$. In Lemma \ref{Lemaaucx:1} we provide an alternative expression for the terms $V_{i}(\varepsilon)$, $i=1,2,3$ defined in \eqref{Vdef}. In Lemma \ref{Local_non_determinism} we prove some useful bounds that we will use later to estimate the covariance of $p_{\varepsilon}(B_{t_{1}}-B_{s_{1}})$ and $p_{\varepsilon}(B_{t_{2}}-B_{s_{2}})$, $s_{1}\leq t_{1}$, $s_{2}\leq t_{2}$ and $s_{1}\leq s_{2}$. In Lemmas \ref{Lema:aux_region1c} and \ref{Lema1:aux_region2} we estimate the order of $V_{1}(\varepsilon)$ and $V_{2}(\varepsilon)$ when $\varepsilon$ is small, while in Lemmas \ref{Lema1:aux_region3:1/12/08/2015} and \ref{Lema2:aux_region3} we determine the exact behavior of $V_{3}(\varepsilon)$ and $V_{3}^{(1)}(\varepsilon)$ as $\varepsilon\rightarrow0$. Finally, we prove Lemmas \ref{Lema:aux_regionsqg2s} and \ref{Lema:aux_regionlqg2}, which were used in Lemma \ref{lemma:varianza:28/04} to determine the behavior of the variance of $I_{2q-1}\left(f_{2q-1,\varepsilon}\right)$ for $q\geq2$.\\

In what follows, $I$ will denote the identity matrix of dimension 2. In addition, for every square matrix $A$ of dimension 2, we will denote by $|A|$ its determinant.
\begin{Lema}\label{Lemaaucx:1}
Let $\varepsilon>0$ be fixed. Define $\Sc_{1}$, $\Sc_{2}$, $ \Sc_{3}$ by \eqref{eq:2:regions}, \eqref{eq:3:regions}, \eqref{eq:4:regions} respectively, and $V_{1}(\varepsilon)$, $V_{2}(\varepsilon)$, $V_{3}(\varepsilon)$ by \eqref{Vdef}.  Then, for $i=1,2,3$, we have
\begin{align}\label{eq1:04/10}
V_{i}(\varepsilon)
  &= \frac{1}{\pi}\int_{\Sc_{i}}\left|\varepsilon I+\Sigma\right|^{-\frac{3}{2}}\Sigma_{1,2}\text{d}s_{1}\text{d}s_{2}\text{d}t_{1}\text{d}t_{2},
\end{align}
 where $\Sigma=(\Sigma_{i,j})_{i,j=1,2}$ is the covariance matrix of $(B_{t_{1}}-B_{s_{1}},B_{t_{2}}-B_{s_{2}})$.
\end{Lema}
\begin{proof}
Let $(X,Y)$ be a jointly Gaussian vector with mean zero, covariance $\Sigma=(\Sigma_{i,j})_{i,j=1,2}$, and density $f_{\Sigma}(x,y)$. First we prove that for every $\theta>0$, 
\begin{align}\label{Lema:1:varianza}
\E\left[XYp_{\theta}(X)p_{\theta}(Y)\right]
  &=(2\pi)^{-1}\theta^{2}\left|\theta I + \Sigma\right|^{-\frac{3}{2}}\Sigma_{1,2}.
\end{align}
To prove this, notice that
\begin{align}\label{eq1:03/09}
\E\left[XYp_{\theta}(X)p_{\theta}(Y)\right]
  &=\int_{\R^{2}}xyp_{\theta}(x)p_{\theta}(y)f_{\Sigma}(x,y)\text{d}x\text{d}y\nonumber\\
  &=(2\pi)^{-2}\theta^{-1}|\Sigma|^{-\frac{1}{2}}\int_{\R^{2}}xy\exp\left\{-\frac{1}{2}(x,y)\left(\theta^{-1}I+\Sigma^{-1}\right)(x,y)^{T}\right\}\text{d}x\text{d}y\nonumber\\
	&=(2\pi)^{-1}\theta^{-1}|\Sigma|^{-\frac{1}{2}}\left|\theta^{-1}I+\Sigma^{-1}\right|^{-\frac{1}{2}}\int_{\R^{2}}xyf_{\widetilde{\Sigma}}(x,y)\text{d}x\text{d}y,
\end{align}
where $\widetilde{\Sigma}:=\left(\theta^{-1}I + \Sigma^{-1}\right)^{-1}$ and $f_{\widetilde{\Sigma}}(x,y)$ denotes the density of a Gaussian vector with   mean zero and covariance $\widetilde{\Sigma}$. Clearly, $\theta^{-1}|\Sigma|^{-\frac{1}{2}}|\theta^{-1}I+\widetilde{\Sigma}^{-1}|^{-\frac{1}{2}}=|\theta I+\Sigma|^{-\frac{1}{2}}$. Then, substituting this identity in \eqref{eq1:03/09}, we get 
\begin{align*}
\E\left[XYp_{\theta}(X)p_{\theta}(Y)\right]
  &=(2\pi)^{-1}\left|\theta I + \Sigma\right|^{-\frac{1}{2}}\int_{\R^{2}}xyf_{\widetilde{\Sigma}}(x,y)\text{d}x\text{d}y\\
	&=(2\pi)^{-1}\left|\theta I + \Sigma\right|^{-\frac{1}{2}}\widetilde{\Sigma}_{1,2}.
\end{align*}
Taking into account that $\widetilde{\Sigma}_{1,2}$ is given by
$$\widetilde{\Sigma}_{1,2}=\theta^{2}|\theta I+\Sigma|^{-1}\Sigma_{1,2},$$ 
we conclude that	
\begin{align*}
\E\left[XYp_{\theta}(X)p_{\theta}(Y)\right]
  &=(2\pi)^{-1}\theta^{2}\left|\theta Id + \Sigma\right|^{-\frac{3}{2}}\Sigma_{1,2},
\end{align*}
as required. 	From \eqref{Lema:1:varianza}, we can write 
\begin{align*}
V_{i}(\varepsilon)
  &= 2\int_{\Sc_{i}}\E\left[p^{\prime}_{\varepsilon}(B_{t_{1}}-B_{s_{1}})p^{\prime}_{\varepsilon}(B_{t_{2}}-B_{s_{2}})\right]\text{d}s_{1}\text{d}s_{2}\text{d}t_{1}\text{d}t_{2}\nonumber\\
	&= \frac{2}{\varepsilon^{2}}\int_{\Sc_{i}}\E\left[(B_{t_{1}}-B_{s_{1}})(B_{t_{2}}-B_{s_{2}}) p_{\varepsilon}(B_{t_{1}}-B_{s_{1}}) p_{\varepsilon}(B_{t_{2}}-B_{s_{2}})\right]\text{d}s_{1}\text{d}s_{2}\text{d}t_{1}\text{d}t_{2}\nonumber\\
	&= \frac{1}{\pi}\int_{\Sc_{i}}\left|\varepsilon I+\Sigma\right|^{-\frac{3}{2}}\Sigma_{1,2}\text{d}s_{1}\text{d}s_{2}\text{d}t_{1}\text{d}t_{2}.
\end{align*}
This finishes the proof of \eqref{eq1:04/10}.
\end{proof}
\begin{Lema}\label{Local_non_determinism}
Let $s_{1},s_{2},t_{1},t_{2}\in\R_{+}$ be such that $s_{1}\leq s_{2}$, and $s_{i}\leq t_{i}$  for $i=1,2$. Denote by $\Sigma$ the covariance matrix of $(B_{t_{1}}-B_{s_{1}},B_{t_{2}}-B_{s_{2}})$. Then, if $s_{1}< s_{2}< t_{2}< t_{1}$, or $s_{1}<t_{1}<s_{2}<t_{2}$, there exists $0<\delta<1$ such that 
\begin{align}\label{eq8:02/04/15}
\left|\Sigma\right|
  \geq \delta (t_{1}-s_{1})^{2H}(t_{2}-s_{2})^{2H}.
\end{align}
In addition, if $s_{1}< s_{2}< t_{2}< t_{1}$, then  exists $0<\delta<1$ such that
\begin{align}\label{eq800000:02/04/15}
\left|\Sigma\right|
  \geq \delta ((a+b)^{2H}c^{2H}+(b+c)^{2H}a^{2H}),
\end{align}
where $a:=s_{2}-s_{1}$, $b:=t_{1}-s_{2}$, and $c:=t_{2}-t_{1}$.
\end{Lema}
\begin{proof}
The result follows from the local non-determinism property of the fractional Brownian motion (see \cite{HuNu}, Lemma 9).
\end{proof}
\begin{Lema}\label{Lema:aux_region1c}
Let $\varepsilon>0$ and define $V_{1}(\varepsilon)$ by \eqref{Vdef}. Then, for every $\frac{2}{3}<H<1$ we have 
\begin{align}
\lim_{\varepsilon\rightarrow0}\varepsilon^{3-\frac{2}{H}}V_{1}(\varepsilon)
   &=0\label{eq:1:30/04/2015}.
\end{align}
\end{Lema}
\begin{proof}
Changing the coordinates $(s_{1},s_{2},t_{1},t_{2})$ by $(s_{1},a:=s_{2}-s_{1},b:=t_{1}-s_{2},c:=t_{2}-t_{1})$ in \eqref{eq1:04/10}, we get
\begin{align}\label{eq1fdvfd:04/10:16/10}
V_{1}(\varepsilon)
	&\leq \frac{1}{\pi}\int_{[0,T]^{4}}\left| \varepsilon I + \Sigma\right|^{-\frac{3}{2}}\Sigma_{1,2}\text{d}s_{1}\text{d}a\text{d}b\text{d}c,
\end{align}
where  $\Sigma$ denotes the covariance matrix of $(B_{a+b},B_{a+b+c}-B_{a})$, namely, 
\begin{align}
\Sigma_{1,1}
  &=(a+b)^{2H},\label{eq:160/10/2014}\\
\Sigma_{2,2}
  &=(c+b)^{2H},\label{eq:161/10/2014}\\
\Sigma_{1,2}
  &=\frac{1}{2}((a+b+c)^{2H}+b^{2H}-c^{2H}-a^{2H})\label{eq:162/10/2014}.
\end{align} 
Integrating the $s_{1}$ variable in \eqref{eq1fdvfd:04/10:16/10} we obtain 
\begin{align}\label{eq1s:04/10:16/10}
V_{1}(\varepsilon)
	&\leq \frac{T}{\pi}\int_{[0,T]^{3}}\left| \varepsilon I + \Sigma\right|^{-\frac{3}{2}}\Sigma_{1,2}\text{d}a\text{d}b\text{d}c.
\end{align}
Next we bound the right-hand side of \eqref{eq1s:04/10:16/10}. Applying \eqref{eq800000:02/04/15}, \eqref{eq:160/10/2014}, \eqref{eq:161/10/2014} and \eqref{eq:162/10/2014}, we get 
\begin{align}\label{eq:2/04/02/15}
\left|\varepsilon I+\Sigma\right|
  &=(\varepsilon+\Sigma_{1,1})(\varepsilon+\Sigma_{2,2})-\Sigma_{1,2}^{2}=\varepsilon^{2}+\varepsilon \Sigma_{1,1}+\varepsilon\Sigma_{2,2}+|\Sigma|\nonumber\\
	&\geq \delta(\varepsilon^2+\varepsilon(a+b)^{2H}+\varepsilon(b+c)^{2H}+(a+b)^{2H}c^{2H}+(b+c)^{2H}a^{2H}),
\end{align}
for some $\delta>0$ only depending on $H$. Using the inequality $\Sigma_{1,2}\leq(a+b)^{H}(b+c)^{H}$, as well as \eqref{eq1s:04/10:16/10} and \eqref{eq:2/04/02/15}, we deduce that there exists a constant $K$ only depending on $T,H$ such that 
\begin{align}\label{eq1:04/10:16/10}
V_{1}(\varepsilon)
	&\leq K\int_{[0,T]^{3}}\frac{(a+b)^{H}(b+c)^{H}}{\Theta_{\varepsilon}(a,b,c)^{\frac{3}{2}}}\text{d}a\text{d}b\text{d}c,
\end{align}
where the function $\Theta_{\varepsilon}$ is defined by 
\begin{align}\label{eq1n:21/0/2015}
\Theta_{\varepsilon}(a,b,c)
	&:= \varepsilon^2+\varepsilon(a+b)^{2H}+\varepsilon(b+c)^{2H}+c^{2H}(a+b)^{2H}+a^{2H}(b+c)^{2H}.
\end{align}
By the arithmetic mean-geometric mean inequality, we have 
$$\frac{1}{2}((a+b)^{2H}+(b+c)^{2H})\geq(a+b)^{H}(b+c)^{H},$$ 
and 
$$\frac{1}{2}(c^{2H}(a+b)^{2H}+a^{2H}(b+c)^{2H})\geq (a+b)^{H}(b+c)^{H}(ac)^{H}.$$
Consequently, 
\begin{align*}
\Theta_{\varepsilon}
  &\geq 2(a+b)^{H}(b+c)^{H}(\varepsilon+(ac)^{H}).
\end{align*}
Therefore, by \eqref{eq1:04/10:16/10} there exists a constant $K>0$ only depending on $T$ and $H$ such that 
\begin{align}\label{eq18}
V_{1}(\varepsilon)
	&\leq K\int_{[0,T]^{3}}(a+b)^{-\frac{H}{2}}(b+c)^{-\frac{H}{2}}(\varepsilon+(ac)^{H})^{-\frac{3}{2}}\text{d}a\text{d}b\text{d}c\nonumber\\
	&\leq K\int_{[0,T]^{3}}b^{-H}(\varepsilon+(ac)^{H})^{-\frac{3}{2}}\text{d}a\text{d}b\text{d}c.
\end{align}
Let $0<y<\frac{3H}{2}-1$ be fixed, and define $\gamma:=\frac{2y}{3H}+1-\frac{2}{3H}$. By the weighted arithmetic mean-geometric mean inequality, we have 
\begin{align*}
\gamma\varepsilon+(1-\gamma)(ac)^{H}
  &\geq \varepsilon^{\gamma}(ac)^{(1-\gamma)H}.
\end{align*}
Hence, by \eqref{eq18}, we get 
\begin{align*}
\varepsilon^{3-\frac{2}{H}}V_{1}(\varepsilon)
	&\leq K\varepsilon^{3-\frac{2}{H}-\frac{3\gamma}{2}}\int_{[0,T]^{3}}b^{-H}(ac)^{-\frac{3}{2}(1-\gamma)H}\text{d}a\text{d}b\text{d}c\\
	&=  K\varepsilon^{\frac{3}{2}-\frac{1}{H}-\frac{y}{H}}\left(\int_{0}^{T}b^{-H}\text{d}b\right)\left(\int_{[0,T]^{2}}(ac)^{-1+y}\text{d}a\text{d}c\right).
\end{align*}
This implies that  \eqref{eq:1:30/04/2015} holds and the proof of the lemma is complete.
\end{proof}


\begin{Lema}\label{Lema1:aux_region2}
Let $\varepsilon>0$ be fixed. Define $V_{2}(\varepsilon)$ by \eqref{Vdef}. Then, for every $\frac{2}{3}<H<1$,  
\begin{align}
\lim_{\varepsilon\rightarrow0}\varepsilon^{3-\frac{2}{H}}V_{2}(\varepsilon)
	&=0.\label{eq:732ewq:07/02}
\end{align}
\end{Lema}
\begin{proof}
Changing the coordinates $(s_{1},s_{2},t_{1},t_{2})$ by $(s_{1},a:=s_{2}-s_{1},b:=t_{2}-s_{2},t_{1}-t_{2})$ in \eqref{eq1:04/10} for $i=2$, and integrating $s_{1}$, we obtain, as before
\begin{align}\label{eq1sr2:04/10:16/10}
V_{2}(\varepsilon)
	&\leq \frac{T}{\pi}\int_{[0,T]^{3}}\left| \varepsilon I + \Sigma\right|^{-\frac{3}{2}}\Sigma_{1,2}\text{d}a\text{d}b\text{d}c,
\end{align}
where the matrix $\Sigma$ is given by
\begin{align}
\Sigma_{1,1}
  &=(a+b+c)^{2H},\label{eq:160r2/10/2014}\\
\Sigma_{2,2}
  &=b^{2H},\label{eq:161r2/10/2014}\\
\Sigma_{1,2}
  &=\frac{1}{2}((a+b)^{2H}+(b+c)^{2H}-c^{2H}-a^{2H})\label{eq:162r2/10/2014}.
\end{align} 
Using relation \eqref{eq8:02/04/15} in Lemma \ref{Local_non_determinism}, as well as \eqref{eq:160r2/10/2014}, \eqref{eq:161r2/10/2014} and \eqref{eq:162r2/10/2014}, we get 
\begin{align}\label{eq:2r2/04/02/15}
\left|\varepsilon I+\Sigma\right|
  &=(\varepsilon+\Sigma_{1,1})(\varepsilon+\Sigma_{2,2})-\Sigma_{1,2}^{2}=\varepsilon^{2}+\varepsilon(\Sigma_{1,1}+\Sigma_{2,2})+|\Sigma|\nonumber\\
  &\geq \varepsilon^2+\varepsilon((a+b+c)^{2H}+b^{2H})+\delta(a+b+c)^{2H}b^{2H}.
\end{align}
From \eqref{eq1sr2:04/10:16/10} and \eqref{eq:2r2/04/02/15} we deduce that there exists a constant $K>0$, only depending on $T$ and $H$, such that 
\begin{align}\label{eq1:02/09/2015}
V_{2}(\varepsilon)
  &\leq K\int_{[0,T]^{3}}\frac{\Sigma_{1,2}}{(\varepsilon^{2}+\varepsilon(b^{2H}+(a+b+c)^{2H})+b^{2H}(a+b+c)^{2H})^{\frac{3}{2}}}\text{d}a\text{d}b\text{d}c.
\end{align}
The term $\Sigma_{1,2}$ can be written as
\begin{align*}
\Sigma_{1,2}
  &= \frac{1}{2}\left((a+b)^{2H}+(b+c)^{2H}-a^{2H}-c^{2H}\right)\\
	&= 	Hb\int_{0}^{1}\left((a+bv)^{2H-1}+(c+bv)^{2H-1}\right)\text{d}v,
\end{align*}
 which implies
\begin{align}
\Sigma_{1,2}
  &\leq 2Hb(a+b+c)^{2H-1}\label{eq61s:01/07/2015}.
\end{align}
From \eqref{eq1:02/09/2015} and \eqref{eq61s:01/07/2015}, we deduce that there exists a constant $K>0$ only depending on $T$ and $H$, such that 
\begin{align}\label{eq2:02/09/2015}
V_{2}(\varepsilon)
  &\leq K\int_{[0,T]^{3}}\frac{b(a+b+c)^{2H-1}}{(\varepsilon^{2}+\varepsilon(b^{2H}+(a+b+c)^{2H})+b^{2H}(a+b+c)^{2H})^{\frac{3}{2}}}\text{d}a\text{d}b\text{d}c.
\end{align}
Therefore, using the inequality 
\begin{align*}
(\varepsilon^{2}+\varepsilon(b^{2H}+(a+b+c)^{2H})+b^{2H}(a+b+c)^{2H})^{\frac{3}{2}}
  &\geq (\varepsilon(a+b+c)^{2H}+b^{2H}(a+b+c)^{2H})^{\frac{3}{2}},
\end{align*}
we get 
\begin{align}\label{eq2:02/09/20152}
V_{2}(\varepsilon)
  &\leq K\int_{[0,T]^{3}}(a+b+c)^{-(H+1)}b(\varepsilon+b^{2H})^{-\frac{3}{2}}\text{d}a\text{d}b\text{d}c\nonumber\\
	&\leq K\left(\int_{[0,T]^{2}}(a+c)^{-(H+1)}\text{d}a\text{d}c\right)\left(\int_{0}^{T}b(\varepsilon+b^{2H})^{   -\frac{3}{2}}\text{d}b\right).
\end{align}
The term $(a+c)^{-(H+1)}$ is clearly integrable over the region $0\leq a,c\leq T$. To bound the integral over $0\leq b\leq T$ of $b(\varepsilon+b^{2H})^{-\frac{3}{2}}$ we proceed as follows. Define $y:=\frac{3}{2}-\frac{1}{H}$. Notice that $0<y<1$ due to the condition $\frac{2}{3}<H<1$.
%
Therefore, by the weighted arithmetic mean-geometric mean inequality, we have
\begin{align}\label{eq3n:21afaf/08/2015}
y\varepsilon+(1-y)b^{2H}
&\geq \varepsilon^{y}b^{2H(1-y)}.
\end{align}
From \eqref{eq2:02/09/20152} and \eqref{eq3n:21afaf/08/2015}, it follows that there exists a constant $K>0$, only depending on $H$ and $T$, such that 
\begin{align}
\varepsilon^{3-\frac{2}{H}}V_{2}(\varepsilon)
  &\leq K\varepsilon^{3-\frac{2}{H}-\frac{3y}{2}}\int_{0}^{T}b^{1-3H(1-y)}\text{d}b\nonumber\\
	&=    K\varepsilon^{\frac{3}{4}-\frac{1}{2H}}\int_{0}^{T}b^{\frac{3H}{2}-2}\text{d}b.\label{eq3n:21afafaaaadsbjjhjhj/08/2015}
\end{align}
The integral in the right-hand side of \eqref{eq3n:21afafaaaadsbjjhjhj/08/2015} is finite thanks to the condition $H>\frac{2}{3}$. Relation \eqref{eq:732ewq:07/02} then follows by taking limit as $\varepsilon\rightarrow0$ in \eqref{eq3n:21afafaaaadsbjjhjhj/08/2015}.
\end{proof}
\begin{Lema}\label{beta}
Let $c$, $\beta$, $\alpha$ and $\gamma$ be real numbers  such that $c$, $\beta>0$, $\alpha>-1$ and $1+\alpha+\gamma\beta<0$. Then we have
\begin{align}\label{eq31:02/07}
\int_{0}^{\infty}a^{\alpha}(c+a^{\beta})^{\gamma}\text{d}a
  &=\beta^{-1}c^{\frac{\alpha+1+\beta\gamma}{\beta}}B\left(\frac{\alpha+1}{\beta},-\frac{1+\alpha+\gamma\beta}{\beta}\right),
\end{align}
where $B\left(\cdot,\cdot\right)$ denotes the Beta function.
\end{Lema}
\begin{proof}
Making the change of variables $x=a^{\beta}$ in the left-hand side of \eqref{eq31:02/07} we obtain
\begin{align}\label{eq32:02/07}
\int_{0}^{\infty}a^{\alpha}(c+a^{\beta})^{\gamma}\text{d}a
  &=\beta^{-1}\int_{0}^{\infty}x^{\frac{\alpha+1-\beta}{\beta}}(c+x)^{\gamma}\text{d}x.
\end{align}
Hence, making the change of variables $a=\frac x c$ in the right hand side of \eqref{eq32:02/07} we get 
\begin{align}\label{eq33:02/07}
\int_{0}^{\infty}a^{\alpha}(c+a^{\beta})^{\gamma}\text{d}a
  &=\beta^{-1}c^{\frac{\alpha+1+\beta\gamma}{\beta}}\int_{0}^{\infty}a^{\frac{\alpha+1-\beta}{\beta}}(1+a)^{\gamma}\text{d}a.
\end{align}
Finally, the change of variables $x=\frac{a}{1+a}$ in the right hand side of \eqref{eq33:02/07} leads to 
\begin{align}\label{eq34:02/07}
\int_{0}^{\infty}a^{\alpha}(c+a^{\beta})^{\gamma}\text{d}a
  &=\beta^{-1}c^{\frac{\alpha+1+\beta\gamma}{\beta}}\int_{0}^{1}x^{\frac{\alpha+1-\beta}{\beta}}(1-x)^{-\frac{\beta+1+\alpha+\gamma\beta}{\beta}}\text{d}x,
\end{align}
which implies the desired result.
\end{proof}

\begin{Lema}\label{Lema1:aux_region3:1/12/08/2015}
Let $\varepsilon,T>0$, and define $V_{3}(\varepsilon)$ by \eqref{Vdef}. Then, for every $\frac{2}{3}<H<1$ we have
\begin{align}
\lim_{\varepsilon\rightarrow0}\varepsilon^{3-\frac{2}{H}}V_{3}(\varepsilon)
  &	=	\sigma^{2}\label{eq1:24/02/2015},
\end{align}
where $\sigma^{2}$ is given by \eqref{eq:5:01/04/2015}.
\end{Lema}
\begin{proof}
Changing the coordinates $(x,u_{1},u_{2})$ by $(a:=u_{1},b:=x-u_{1},c:=u_{2})$ in \eqref{eq1:04/10} for $i=3$, we obtain
\begin{align}\label{eq1sr23:04/10:16/10}
V_{3}(\varepsilon)
	&= \frac{1}{\pi}\int_{[0,T]^{3}}\Indi{(0,T)}(a+b+c)(T-(a+b+c))\left| \varepsilon I + \Sigma\right|^{-\frac{3}{2}}\Sigma_{1,2}\text{d}a\text{d}b\text{d}c,
\end{align}
where the matrix $\Sigma$ is given by
\begin{align*}
\Sigma_{1,1}
  &=a^{2H},\\
\Sigma_{2,2}
  &=c^{2H},\\
\Sigma_{1,2}
  &=\frac{1}{2}((a+b+c)^{2H}+b^{2H}-(b+c)^{2H}-(a+b)^{2H}).
\end{align*} 
We can easily check, as before,  that 
\begin{align}\label{eq3:2r2/04/02/15}
\left|\varepsilon I+\Sigma\right|
  &=(\varepsilon+\Sigma_{1,1})(\varepsilon+\Sigma_{2,2})-\Sigma_{1,2}^{2}
  =\varepsilon^{2}+\varepsilon(\Sigma_{1,1}+\Sigma_{2,2})+|\Sigma|\nonumber\\
  &= \varepsilon^2+\varepsilon(a^{2H}+c^{2H})+a^{2H}c^{2H}-\mu(a+b,a,c)^{2},
\end{align}
where $\mu$ is defined by \eqref{mudef}. Changing the coordinates $(a,b,c)$ by $(\varepsilon^{-\frac{1}{2H}}a,b,\varepsilon^{-\frac{1}{2H}}c)$ in \eqref{eq1sr23:04/10:16/10} and using \eqref{eq3:2r2/04/02/15},  we obtain 
\begin{align}\label{eq1sr23prime:04/10:16/10}
\varepsilon^{3-\frac{2}{H}}V_{3}(\varepsilon)
	&= \frac{1}{\pi}\int_{\R_{+}^{3}}\Indi{(0,T)}(\varepsilon^{\frac{1}{2H}}(a+c)+b)\Psi_{\varepsilon}(a,b,c)\text{d}a\text{d}b\text{d}c,
\end{align}
where 
\begin{align*}
\Psi_{\varepsilon}(a,b,c)
	&:= \frac{(T-b-\varepsilon^{\frac{1}{2H}}(a+c))\varepsilon^{-\frac{1}{H}}\mu(\varepsilon^{\frac{1}{2H}}a+b,\varepsilon^{\frac{1}{2H}}a,\varepsilon^{\frac{1}{2H}}c)}{\left( 1+a^{2H}+c^{2H}+a^{2H}c^{2H}-\varepsilon^{-2}\mu(\varepsilon^{\frac{1}{2H}}a+b,\varepsilon^{\frac{1}{2H}}a,\varepsilon^{\frac{1}{2H}}c)^{2} \right)^{\frac{3}{2}}}.
\end{align*}
The term $\mu(x+y,x,z)$ can be written as
\begin{align}
\mu(x+y,x,z)
	&= 	H(2H-1)xz\int_{[0,1]^{2}}\left(b+xv_{1}+zv_{2}\right)^{2H-2}\text{d}v_{1}\text{d}v_{2}\label{eq61dsd:01/07/2015},
\end{align}
which implies
\begin{align}\label{eq10:04/09/2015}
\lim_{\varepsilon\rightarrow}\Psi_{\varepsilon}(a,b,c)
  &= \frac{H(2H-1)(T-b)acb^{2H-2}}{(1+a^{2H}+c^{2H}+a^{2H}c^{2H})^{\frac{3}{2}}}\nonumber\\
	&= (T-b)b^{2H-2}ac(1+a^{2H})^{-\frac{3}{2}}(1+c^{2H})^{-\frac{3}{2}}.
\end{align}
Therefore, provided we show that $\Indi{(0,T)}(\varepsilon^{\frac{1}{2H}}(a+c)+b)\Psi_{\varepsilon}(a,b,c)$ is dominated by a function integrable in $\R_+^3$, we obtain the following identity by applying the dominated convergence theorem in  \eqref{eq1sr23prime:04/10:16/10}
\begin{align*}
\lim_{\varepsilon\rightarrow0}\varepsilon^{3-\frac{2}{H}}V_{3}(\varepsilon)
	&= \frac{H(2H-1)}{\pi}\int_{\R_{+}^{3}}\Indi{(0,T)}(b)(T-b)b^{2H-2}ac((1+a^{2H})(1+c^{2H}))^{-\frac{3}{2}}\text{d}a\text{d}b\text{d}c.
\end{align*}
Making the change of variables $x=\frac{b}{T}$, and using Lemma \ref{beta} we obtain \eqref{eq1:24/02/2015}. Next we show that $\Indi{(0,T)}(\varepsilon^{\frac{1}{2H}}(a+c)+b)\Psi_{3,0,\varepsilon}(a,b,c)$ is dominated by a function integrable in $\R_+^3$. Using  \eqref{eq61dsd:01/07/2015}, we deduce that there exists a constant $K>0$ only depending on $T$ and $H$ such that
\begin{align*}
\Psi_{3,0,\varepsilon}(a,b,c)
	&\leq K\frac{acb^{2H-2}}{(1+a^{2H}+c^{2H}+a^{2H}c^{2H})^{\frac{3}{2}}}\nonumber\\
	&= K  b^{2H-2}ac(1+a^{2H})^{-\frac{3}{2}}(1+c^{2H})^{-\frac{3}{2}}.
\end{align*}
The right-hand side in the previous relation is integrable in $\R_{+}^{3}$ thanks to condition $H>\frac{2}{3}$. The proof is now complete.
\end{proof}
\begin{Lema}\label{Lema2:aux_region3}
Let  $T,\varepsilon>0$ be fixed. Define $V_{3}^{(1)}(\varepsilon)$ by \eqref{Vqdef}. Then, for every $\frac{2}{3}<H<1$ it holds 
\begin{align}
\lim_{\varepsilon\rightarrow0}\varepsilon^{3-\frac{2}{H}}V_{3}^{(1)}(\varepsilon)
  &	=	\sigma^{2},\label{eq6:26/02/2015}
\end{align}
where $\sigma^{2}$ is given by \eqref{eq:5:01/04/2015}.
\end{Lema}
\begin{proof}
By \eqref{Vqdef} and \eqref{eq2a:06/08/2015}, 
\begin{align}\label{eq1sssdd:04/10}
V_{3}^{(1)}(\varepsilon)
  &=(2q-1)!\beta_{q}^{2}\int_{\Sc_{3}}G_{\varepsilon,s_{2}-s_{1}}^{(q)}(t_{1}-s_{1},t_{2}-s_{2}),
\end{align}
where $\Sc_{3}$ is defined by \eqref{eq:4:regions}. Changing the coordinates $(s_{1},s_{2},t_{1},t_{2})$ by $(a:=t_{1}-s_{1},b:=s_{2}-t_{1},c:=t_{2}-s_{2})$ in \eqref{eq1sssdd:04/10}, and using \eqref{Gdef}, we obtain

\begin{eqnarray} \notag
V_{3}^{(1)}(\varepsilon)
 &=& \frac{1}{\pi}\int_{\R_{+}^{3}}\int_{0}^{T-(a+b+c)}\Indi{(0,T)}(a+b+c)\left(\varepsilon+a^{2H}\right)^{-\frac{3}{2}}\left(\varepsilon+c^{2H}\right)^{-\frac{3}{2}}\\
 && \times \mu(a+b,a,c)\text{d}s_{1}\text{d}a\text{d}b\text{d}c. \label{eq21n:09/08/2015}
\end{eqnarray}
Then, changing the coordinates $(a,b,c)$ by  $(\varepsilon^{-\frac{1}{2H}}a,b,\varepsilon^{-\frac{1}{2H}}c)$, and integrating $s_{1}$ in equation \eqref{eq21n:09/08/2015}, we get
\begin{align*}
V_{3}^{(1)}(\varepsilon)
 &=\frac{\varepsilon^{\frac{1}{H}-3}}{\pi}\int_{\R_{+}^{3}}(T-b-\varepsilon^{\frac{1}{2H}}(a+c))\Indi{(0,\varepsilon^{-\frac 1{2H}}(T-b))}(a+c)\\
&\times\left(1+a^{2H}\right)^{-\frac{3}{2}}\left(1+c^{2H}\right)^{-\frac{3}{2}}\mu(\varepsilon^{\frac{1}{2H}}a+b,\varepsilon^{\frac{1}{2H}}a,\varepsilon^{\frac{1}{2H}}c)\text{d}a\text{d}b\text{d}c.
\end{align*}
Next, using the identity 
\begin{align*}
\mu(x+y,x,z)
  &=H(2H-1)xz\int_{[0,1]^{2}}(y+xv_{1}+zv_{2})^{2H-2}\text{d}v_{1}\text{d}v_{2}, 
\end{align*}
we get 
\begin{align}\label{eq:5546:19/02/2015}
\varepsilon^{3-\frac{2}{H}}V_{3}^{(1)}(\varepsilon)  \nonumber
&=\frac{H(2H-1)}{\pi}\int_{0}^{T}\int_{\R_{+}^{2}}\int_{[0,1]^2}\Indi{(0,\varepsilon^{-\frac{1}{2H}}(T-b))}(a+c)(T-b-\varepsilon^{\frac{1}{2H}}(a+c))\\
&\times (1+a^{2H})^{-\frac{3}{2}}(1+c^{2H})^{-\frac{3}{2}}ac(b+\varepsilon^{\frac{1}{2H}}(av_{1}+cv_{2}))^{2H-2}\text{d}v_{1}\text{d}v_{2}\text{d}a\text{d}c\text{d}b.
\end{align}
Notice that the argument of the integral in the right-hand side of \eqref{eq:5546:19/02/2015} is dominated by the function 
\begin{align*}
\Theta(a,b,c,v_{1},v_{2})
  &:=\frac{TH(2H-1)}{\pi}(1+a^{2H})^{-\frac{3}{2}}(1+c^{2H})^{-\frac{3}{2}}acb^{2H-2}.
\end{align*}
The integral $\int_{0}^{T}\int_{\R_{+}^{2}}\int_{[0,1]^2}\Theta(a,b,c,v_{1},v_{2})\text{d}v_{1}\text{d}v_{2}\text{d}a\text{d}c\text{d}b$ is finite thanks to condition $H>\frac{2}{3}$. Therefore, applying the dominated convergence theorem to \eqref{eq:5546:19/02/2015}, we get
\begin{align*}
\lim_{\varepsilon\rightarrow0}\varepsilon^{3-\frac{2}{H}}V_{3}^{(1)}(\varepsilon)
&=\frac{H(2H-1)}{\pi}\int_{0}^{T}\int_{\R_{+}^{2}}(T-b)(1+a^{2H})^{-\frac{3}{2}}(1+c^{2H})^{-\frac{3}{2}}acb^{2H-2}\text{d}a\text{d}c\text{d}b\\
&=\frac{H(2H-1)}{\pi}\left(\int_{0}^{T}(T-b)b^{2H-2}\text{d}b\right)\left(\int_{0}^{\infty}a(1+a^{2H})^{-\frac{3}{2}}\text{d}a\right)^2.
\end{align*}
Making the change of variables $x=\frac{b}{T}$, and using Lemma \ref{beta} we obtain \eqref{eq6:26/02/2015}.
\end{proof}
\begin{Lema}\label{Lema:aux_regionlqg2}
Let $T,\varepsilon>0$ and $q\in\N$, $q\geq2$ be fixed. Define $G_{1,x}^{(q)}(u_{1},u_{2})$ by \eqref{Gdef}. Then, for every $\frac{3}{4}<H<\frac{4q-3}{4q-2}$,  it holds that
\begin{align}
\int_{\R_{+}^{3}}G_{1,x}^{(q)}(u_{1},u_{2})\text{d}x\text{d}u_{1}\text{d}u_{2}<\infty .\label{eq2}	
\end{align}
\end{Lema}
\begin{proof}
Let $T,\varepsilon>0$, and $q\in\N$ be fixed, and define the sets 
\begin{align*}
\Tc_{1}
  &:=\{(x,u_{1},u_{2})\in\R_{+}^{3}\ |\ u_{1}-x\geq0,~x+u_{2}-u_{1}\geq0\},\\
\Tc_{2}
  &:=\{(x,u_{1},u_{2})\in\R_{+}^{3}\ |\ u_{1}-x-u_{2}\geq0\},\\
\Tc_{3}
  &:=\{(x,u_{1},u_{2})\in\R_{+}^{3}\ |\ x-u_{1}\geq0\}.
\end{align*}
Since $\R_{+}^{3}=\Tc_{1}\cup\Tc_{2}\cup \Tc_{3}$, it suffices to prove that $G_{1,x}^{(q)}(u_{1},u_{2})$ is integrable in $\Tc_{i}$, for $i=1,2,3$.

To prove the integrability of $G_{1,x}^{(q)}(u_{1},u_{2})$ in $\Tc_{1}$ we change the coordinates $(x,u_{1},u_{2})$ by $(a:=x,b:=u_{1}-x,c:=x+u_{2}-u_{1})$. Then,
\begin{align}\label{eq1n:09/11/2015}
\int_{\Tc_{1}}G_{1,x}^{(q)}(u_{1},u_{2})\text{d}x\text{d}u_{1}\text{d}u_{2}
  &=\int_{\R_{+}^{3}}G_{1,a}^{(q)}(a+b,b+c)\text{d}a\text{d}b\text{d}c.
\end{align}
Next we prove that the right hand of \eqref{eq1n:09/11/2015} is finite. Notice that 
\begin{align*}
G_{1,a}^{(q)}(a+b,b+c)
  &=(1+(a+b)^{2H})^{-\frac{1}{2}-q}(1+(b+c)^{2H})^{-\frac{1}{2}-q}\mu(a,a+b,b+c)^{2q-1}.
\end{align*}
By the Cauchy-Schwarz inequality, we get $\mu(a,a+b,b+c)\leq(a+b)^{H}(b+c)^{H}$, and consequently,
\begin{align*}
G_{1,a}^{(q)}(a+b,b+c)\leq(1+(a+b)^{2H})^{-1}(1+(b+c)^{2H})^{-1}.
\end{align*}
Hence, using the  inequalities $\frac{2}{3}a+\frac{1}{3}b\geq a^{\frac{2}{3}}b^{\frac{1}{3}}$ and $\frac{2}{3}c+\frac{1}{3}b\geq c^{\frac{2}{3}}b^{\frac{1}{3}}$, we deduce that there exists a constant $K$ only depending on $T$ and  $H$ such that the following bounds hold
\begin{align*}
G_{1,a}^{(q)}(a+b,b+c)
&\leq K(abc)^{-\frac{4H}{3}}~~~~~~~~~~~~~~~~~~~~~~~~~~~~~~~~\text{if }~~a,b,c\geq1,\\
G_{1,a}^{(q)}(a+b,b+c)
&\leq K(1+b^{2H})^{-1}(1+c^{2H})^{-1}~~~~~~~~~~~~~\text{if }~~a\leq1,\\
G_{1,a}^{(q)}(a+b,b+c)
&\leq K(1+b^{2H})^{-1}(1+a^{2H})^{-1}~~~~~~~~~~~~~\text{if }~~c\leq1,\\
G_{1,a}^{(q)}(a+b,b+c)
&\leq K(1+a^{2H})^{-1}(1+c^{2H})^{-1}~~~~~~~~~~~~~\text{if }~~b\leq1.
\end{align*}
Using the previous bounds, as well as condition $H>\frac{3}{4}$, we deduce that $G_{1,a}^{(q)}(a+b,b+c)$ is integrable in the variables $a,b,c\geq0$.

To prove the integrability of $G_{1,x}^{(q)}(u_{1},u_{2})$ in $\Tc_{2}$ we change the coordinates $(x,u_{1},u_{2})$ by $(a:=x,b:=u_{2},c:=u_{1}-x-u_{2})$. Then,
\begin{align*}
\int_{\Tc_{2}}G_{1,x}^{(q)}(u_{1},u_{2})\text{d}x\text{d}u_{1}\text{d}u_{2}
  &=\int_{\R_{+}^{3}}G_{1,a}^{(q)}(b,a+b+c)\text{d}a\text{d}b\text{d}c.
\end{align*}
Next we prove that $G_{1,a}^{(q)}(a+b,b+c)$ is integrable in the variables $a,b,c\geq0$. Using inequality $\mu(a,a+b+c,b)\leq(a+b+c)^{H}b^{H}$, as well as the condition $q\geq2$, we obtain
\begin{align}\label{e1:11/02/15}
G_{1,a}^{(q)}(b,a+b+c)  
  &= (1+(a+b+c)^{2H})^{-\frac{5}{2}}(1+b^{2H})^{-\frac{5}{2}}\mu(x,a+b+c,b)^{3}\nonumber\\
  &\times \left(\frac{\mu(a,a+b+c,b)}{\sqrt{(1+b^{2H})(1+(a+b+c)^{2H})}}\right)^{2(q-2)}\nonumber\\
  &\leq (1\vee a\vee b\vee c)^{-5H}(1\vee b)^{-5H}\mu(a,a+b+c,b)^3.
\end{align}
Similarly, using $q\geq1$ we can prove that 
\begin{align}\label{e1dds:11/02/15}
G_{1,a}^{(q)}(b,a+b+c)  
  &\leq (1\vee a\vee b\vee c)^{-2H}(1\vee b)^{-2H}.
\end{align}
In addition, using the representation 
\begin{align*}
\mu(a,a+b+c,b)
  &= \frac{1}{2}\left((a+b)^{2H}+(b+c)^{2H}-a^{2H}-c^{2H}\right)\nonumber\\
	&= 	Hb\int_{0}^{1}\left((a+bu)^{2H-1}+(c+bu)^{2H-1}\right)\text{d}u,
\end{align*}
we deduce that there exist  constants $K,K^{\prime}$ only depending on $H$ such that 
\begin{align}\label{eq:3/07/02eq:3/07/02}
\mu(a,a+b+c,b)\Indi{(0,a\wedge c)}(b)
  &\leq K\Indi{(0,a\wedge c)}(b)b((a+b)^{2H-1}+(c+b)^{2H-1})\nonumber\\
  &\leq K^{\prime}\Indi{(0,a\wedge c)}(b)b(a\vee c)^{2H-1}\nonumber\\
  &\leq K^{\prime}(1\vee b)(1\vee a\vee c)^{2H-1}.
\end{align}
Combining the inequalities \eqref{e1:11/02/15} and \eqref{eq:3/07/02eq:3/07/02}, we deduce that there exists a constant $K>0$ such that 
\begin{align*}
G_{1,a}^{(q)}(b,a+b+c)\Indi{(0,a\wedge c)}(b)  
  &\leq K\Indi{(0,a\wedge c)}(b)(1\vee a\vee b\vee c)^{-5H}(1\vee b)^{-5H+3}(1\vee a\vee c)^{6H-3}\\
  &\leq K(1\vee a\vee c)^{H-3}(1\vee b)^{-5H+3}.
\end{align*}
Using the previous inequality, as well as the condition $H>\frac{3}{4}$, we deduce that $G_{1,a}^{(q)}(b,a+b+c)$ is integrable in $\{(a,b,c)\in\R_{+}^{3}\ | \ b\leq a\wedge c\}$. 
In addition, from \eqref{e1dds:11/02/15} we obtain
\begin{align*}
G_{1,a}^{(q)}(b,a+b+c)\Indi{(0,b\wedge c)}(a)
  &\leq (1\vee b)^{-2H}(1\vee b\vee c)^{-2H}.
\end{align*}
Therefore, using  condition $H>\frac{3}{4}$, we deduce that $G_{1,a}^{(q)}(b,a+b+c)$ is integrable in $\{(a,b,c)\in\R_{+}^{3}\ |\ a\leq b\wedge c\}$. By symmetry $G_{1,a}^{(q)}(b,a+b+c)$ is integrable  in $\{(a,b,c)\in\R_{+}^{3}\ |\ c\leq a\wedge b\}$. 
From the previous analysis we conclude that $G_{1,x}^{(q)}(u_{1},u_{2})$ is integrable in $\Tc_{2}$.

To prove the integrability of $G_{1,x}^{(q)}(u_{1},u_{2})$ in $\Tc_{3}$, we change the coordinates $(x,u_{1},u_{2})$ by $(a:=u_{1},b:=x-u_{1},c:=u_{2})$. Then, 
\begin{align*}
\int_{\Tc_{3}}G_{1,x}^{(q)}(u_{1},u_{2})\text{d}x\text{d}u_{1}\text{d}u_{2}
  &=\int_{\R_{+}^{3}}G_{1,a+b}^{(q)}(a,c)\text{d}a\text{d}b\text{d}c.
\end{align*}
To bound $G_{1,a+b}^{(q)}(a,c)$ we proceed as follows. Using inequality $\mu(a+b,a,c)\leq a^{H}c^{H}$, we deduce that 
\begin{align}\label{eq1b:13/10/2015}
G_{1,a+b}^{(q)}(a,c)
  &\leq \left(1+a^{2H}\right)^{-1}\left(1+c^{2H}\right)^{-1}\nonumber\\
  &\leq \left(1\vee a\right)^{-2H}\left(1\vee c\right)^{-2H}.
\end{align}
As a consequence, $G_{1,a+b}^{(q)}(a,c)$ is integrable in $\{(a,b,c)\in\R_{+}^{3}\ |\ b\leq 1\}$. In addition, from relation 
\begin{align}\label{eq:1:05/091}
\mu(x+y,x,z)
  &=H(2H-1)xz\int_{[0,1]^{2}}(y+xv_{1}+zv_{2})^{2H-2}\text{d}v_{1}\text{d}v_{2},
\end{align}
we can prove that
\begin{align}\label{eqq:9:02/02}
\mu(x+y,x,z)
  &\leq H(2H-1)xzy^{2H-2}.
\end{align}
Using  \eqref{eqq:9:02/02}, we deduce that there exists a constant $K>0$, only depending on $H$ and $q$, such that  
\begin{align*}
G_{1,a+b}^{(q)}(a,c)
  &\leq K\left(\left(1+a^{2H}\right)\left(1+c^{2H}\right)\right)^{-\frac{1}{2}-q}(ac)^{2q-1}b^{2(2q-1)(H-1)}\\
	&\leq K\left(\left(1\vee a\right)\left(1\vee c\right)\right)^{-H-2qH+2q-1}b^{2(2q-1)(H-1)}.
\end{align*}
Taking into account that $H<\frac{4q-3}{4q-2}$, we get $2(2q-1)(H-1)<-1$, and hence
\begin{align}   \nonumber
\int_{1\vee a\vee c}^{\infty} G_{1,a+b}^{(q)}(a,c)\text{d}b
&  \leq K\left(\left(1\vee a\right)\left(1\vee c\right)\right)^{-H-2qH+2q-1}(1\vee a\vee c)^{2(2q-1)(H-1)+1}\\
  &\leq K\left(1\vee a\right)^{-2H+\frac{1}{2}}\left(1\vee c\right)^{-2H+\frac{1}{2}}, \label{sseq23:03/03/2015}
\end{align}
where in the last inequality we used the relation
\begin{align*}
(1\vee a\vee c)^{2(2q-1)(H-1)+1}\leq (1\vee a)^{(2q-1)(H-1)+\frac{1}{2}}(1\vee c)^{(2q-1)(H-1)+\frac{1}{2}}.
\end{align*}
Using  relation \eqref{sseq23:03/03/2015} as well as condition $H>\frac{3}{4}$, we conclude that $G_{1,a+b}^{(q)}(a,c)$ is integrable in $\{(a,b,c)\in\R_{+}^{3}\ |\ 1\vee a\vee c\leq b\}$. In addition, from \eqref{eq:1:05/091}  we obtain
\begin{align*}
\mu(x+y,x,z)
  &\leq H(2H-1)xz\int_{[0,1]^2}(xv_{1}+zv_{2})^{2H-2}\text{d}v_{1}\text{d}v_{2}\\
	&\leq H(2H-1)xz\int_{0}^{1}((x\vee z)w)^{2H-2}\text{d}w\\
  &=    Hxz(x\vee z)^{2H-2}=    H(x\wedge z)(x\vee z)^{2H-1}.
\end{align*}
Hence, there exist constants $K,\widetilde{K}\geq0$ such that 
\begin{multline}\label{eq2b:13/10/2015}
G_{1,a+b}^{(q)}(a,c)\Indi{(a\wedge c,a\vee c)}(b)\\
\begin{aligned}
  &= \left(\left(1+a^{2H}\right)\left(1+c^{2H}\right)\right)^{-\frac{1}{2}-q}\mu(a+b,a,c)^{2q-1}\\
  &\leq K\left(\left(1\vee a\right)\left(1\vee c\right)\right)^{-H-2qH}(a\wedge c)^{2q-1}(a\vee c)^{(2q-1)(2H-1)}\\
  &\leq K\left(\left(1\vee a\right)\left(1\vee c\right)\right)^{-H-2qH}(1\vee(a\wedge c))^{2q-1}(1\vee a\vee c)^{(2q-1)(2H-1)}\\
  &= K(1\vee (a\wedge c))^{-H(2q+1)+2q-1}(1\vee a\vee c)^{-3H-2q+2qH+1}.
\end{aligned}
\end{multline}
Using relation \eqref{eq2b:13/10/2015} as well as condition $H>\frac{3}{4}$, we obtain that $G_{1,a+b}^{(q)}(a,c)$ is integrable in the 
region $\{(a,b,c)\in\R_{+}^{3}\ |\ a\wedge c\leq b\leq a\vee c\}$. Finally, applying \eqref{eq1b:13/10/2015} we can prove that $G_{1,a+b}^{(q)}(a,c)$ is integrable in $\{(a,b,c)\in\R_{+}^{3}\ |\ b\leq a\wedge c\}$. From the previous analysis we conclude 
that $G_{1,a+b}^{(q)}(a,c)$ is integrable in the variables $a,b,c\geq0$, which in turn implies that $G_{1,x}^{(q)}(u_{1},u_{2})$ is integrable in $\Tc_{3}$ as required.
\end{proof}

\begin{Lema}\label{Lema:aux_regionsqg2s}
Let $T,\varepsilon>0$ and $q\in\N$, $q\geq2$ be fixed, and define $G_{0,x}^{(q)}(u_{1},u_{2})$ by \eqref{Gdef}. Then, for every $\frac{2}{3}<H<\frac{3}{4}$, we have
\begin{align*}
\int_{[0,T]^{3}}G_{0,x}^{(q)}(u_{1},u_{2}) \text{d}x {\text d}u_{1}{\text d}u_{2}<\infty.
\end{align*}
\end{Lema}
\begin{proof}
Let $T,\varepsilon>0$, and $q\in\N$, and define the sets 
\begin{align*}
\widetilde{\Tc}_{1}
  &:=\{(x,u_{1},u_{2})\in[0,T]^{3}\ |\ u_{1}-x\geq0,~x+u_{2}-u_{1}\geq0\},\\
\widetilde{\Tc}_{2}
  &:=\{(x,u_{1},u_{2})\in[0,T]^{3}\ |\ u_{1}-x-u_{2}\geq0\},\\
\widetilde{\Tc}_{3}
  &:=\{(x,u_{1},u_{2})\in[0,T]^{3}\ |\ x-u_{1}\geq0\}.
\end{align*}
Since $[0,T]^{3}=\widetilde{\Tc}_{1}\cup\widetilde{\Tc}_{2}\cup \widetilde{\Tc}_{3}$, it suffices to check the integrability of $G_{0,x}^{(q)}(u_{1},u_{2})$ in $\widetilde{\Tc}_{i}$, for $i=1,2,3$. To prove integrability in $\widetilde{\Tc}_{1}$ we make change the coordinates $(x,u_{1},u_{2})$ by $(a:=x,b:=u_{1}-x,c:=x+u_{2}-u_{1})$. Then, 
\begin{align*}
\int_{\widetilde{\Tc}_{1}}G_{0,x}^{(q)}(u_{1},u_{2})\text{d}x\text{d}u_{1}\text{d}u_{2}
  &\leq\int_{[0,T]^{3}}G_{0,a}^{(q)}(a+b,b+c)\text{d}a\text{d}b\text{d}c.
\end{align*}
By the inequality  $\mu(a,a+b,b+c)\leq(a+b)^{H}(b+c)^{H}$,  we can write
\begin{align}\label{eq1n:07/08/2015}
G_{0,a}^{(q)}(a+b,b+c)
  \leq (a+b)^{-2H}(b+c)^{-2H}.
\end{align}
Therefore, using  $\frac{2a}{3}+\frac{b}{3}\geq a^{\frac{2}{3}}b^{\frac{1}{3}}$ and $\frac{2c}{3}+\frac{b}{3}\geq c^{\frac{2}{3}}b^{\frac{1}{3}}$, as well as \eqref{eq1n:07/08/2015}, we deduce that there exists a universal constant $K$ such that 
\begin{align*}
G_{0,a}^{(q)}(a+b,b+c)
  \leq K(abc)^{-\frac{4H}{3}}.
\end{align*}
The right hand side in the previous inequality is integrable in $[0,T]^{3}$ thanks to the condition $H<\frac{3}{4}$. Therefore, $G_{0,x}^{(q)}(u_{1},u_{2})$ is integrable in $\widetilde{\Tc}_{1}$.

To prove the integrability of $G_{0,x}^{(q)}(u_{1},u_{2})$ in $\widetilde{\Tc}_{2}$ we change the coordinates $(x,u_{1},u_{2})$ by $(a:=x,b:=u_{2},c:=u_{1}-x-u_{2})$. Then,
\begin{align*}
\int_{\widetilde{\Tc}_{2}}G_{0,x}^{(q)}(u_{1},u_{2})\text{d}x\text{d}u_{1}\text{d}u_{2}
  &\leq\int_{[0,T]^{3}}G_{0,a}^{(q)}(b,a+b+c)\text{d}a\text{d}b\text{d}c.
\end{align*}
In order to bound the term $G_{0,a}^{(q)}(b,a+b+c)$ we proceed as follows. Applying the inequality $\mu(a,a+b+c,b)\leq(a+b+c)^{H}b^{H}$, as well as the condition $q\geq2$, we obtain 
\begin{align}\label{se1:11/02/15}
G_{0,a}^{(q)}(b,a+b+c)  
  &= (a+b+c)^{-5H}b^{-5H}\mu(a,a+b+c,b)^3\nonumber\\
	&\times\left(\frac{\mu(b,a+b+c,b)}{b^{H}(a+b+b)^{H}}\right)^{2(q-2)}\nonumber\\
	&\leq (a+b+c)^{-5H}b^{-5H}\mu(a,a+b+c,b)^3.
\end{align}
On the other hand, by the relation
\begin{align*}
\mu(a,a+b+c,b)
  &= \frac{1}{2}\left((a+b)^{2H}+(b+c)^{2H}-a^{2H}-c^{2H}\right)\nonumber\\
	&= 	Hb\int_{0}^{1}\left((a+bw)^{2H-1}+(c+bw)^{2H-1}\right)\text{d}w,
\end{align*}
we deduce that there exists a constant $K>0$  such that 
\begin{align}\label{se1:11/02/1asdads5}
\mu(a,a+b+c,b)\Indi{(0,a\wedge c)}(b)
  &\leq \Indi{(0,a\wedge c)}(b)Hb\int_{0}^{1}\left((a+bw)^{2H-1}+(c+bw)^{2H-1}\right)\text{d}w\nonumber\\
	&= 	  Kb (a\vee c)^{2H-1}.
\end{align}
Using \eqref{se1:11/02/15} and \eqref{se1:11/02/1asdads5} we get 
\begin{align}\label{ddse1:11/02/1asdads5}
G_{0,a}^{(q)}(b,a+b+c) \Indi{(0,a\wedge c)}(b)
  &\leq 	  Kb^{-5H+3}(a+b+c)^{-5H}(a\vee c)^{6H-3}\nonumber\\
	&\leq 	  Kb^{-5H+3}(a\vee c)^{H-3}.
\end{align}
From \eqref{ddse1:11/02/1asdads5} as well as the condition $H<\frac{3}{4}$, we deduce that $G_{0,a}^{(q)}(b,a+b+c)$ is integrable in $\{(a,b,c)\in[0,T]^{3}\ |\ b\leq a\wedge c\}$. In addition, using the relation $\mu(a,a+b+c,b)\leq(a+b+c)^{H}b^{H}$, we can prove that 
\begin{align*}
G_{0,a}^{(q)}(b,a+b+c)
  &\leq b^{-2H}c^{-2H}.
\end{align*}
Therefore, by the condition $H<\frac{3}{4}$, we deduce that $G_{0,a}^{(q)}(b,a+b+c)$ is integrable in $\{(a,b,c)\in[0,T]^{3}\ |\ a\leq b\wedge c\}$. Similarly, we can prove that 
\begin{align*}
G_{0,a}^{(q)}(b,a+b+c)
  &\leq b^{-2H}a^{-2H},
\end{align*}
and hence, since $H<\frac{3}{4}$ we conclude that $G_{0,a}^{(q)}(b,a+b+c)$ is integrable  in $\{(a,b,c)\in[0,T]^{3}\ |\ c\leq b\wedge a\}$. From the analysis we conclude that $G_{0,a}^{(q)}(b,a+b+c)$ is integrable in $[0,T]^{3}$.

To prove the integrability of $G_{0,x}^{(q)}(u_{1},u_{2})$ in $\widetilde{\Tc}_{3}$ we change the coordinates $(x,u_{1},u_{2})$ by  $(a:=u_{1},b:=x-u_{1},c:=u_{2})$ to get
\begin{align*}
\int_{\widetilde{\Tc}_{3}}G_{0,x}^{(q)}(u_{1},u_{2})\text{d}x\text{d}u_{1}\text{d}u_{2}
  \leq \int_{[0,T]^{3}}G_{0,a+b}^{(q)}(a,c)\text{d}a\text{d}b\text{d}c.
\end{align*}
In order to bound the term $G_{0,a+b}^{(q)}(a,c)$ we proceed as follows. From relation 
\begin{align}\label{eqqqq:1:05/091}
\mu(x+y,x,z)
  &=H(2H-1)xz\int_{[0,1]^{2}}(y+xv_{1}+zv_{2})^{2H-2}\text{d}v_{1}\text{d}v_{2},
\end{align}
we can deduce that
\begin{align*}
\mu(x+y,x,z)
  &\leq H(2H-1)xzy^{2H-2}.
\end{align*}
Hence, since 
\begin{align}\label{eq1:11/11/2015}
G_{0,a+b}^{(q)}(a,c)
  &=a^{-2H-2qH}c^{-2H-2qH}\mu(a+b,a,c)^{2q-1},
\end{align}
we deduce that there exists a constant $K>0$ only depending on $H$ such that 
\begin{align}\label{fkndsusdfgsdfgccdfdsds}
G_{0,a+b}^{(q)}(a,c)\Indi{(a\vee c,T)}(b)
  &\leq a^{-H-2qH+2q-1}c^{-H-2qH+2q-1}b^{2(2q-1)(H-1)}\Indi{(a\vee c,T)}(b).
\end{align}
Since $q\geq 2$, we have that $H<\frac{3}{4}<\frac{4}{5}\leq \frac{2q}{1+2q}$. As a consequence, from \eqref{fkndsusdfgsdfgccdfdsds} we deduce that $G_{0,a+b}^{(q)}(a,c)$ is integrable in $\{(a,b,c)\in\R_{+}^{3}\ |\ b\geq a,c\}$.
In addition, by \eqref{eqqqq:1:05/091} we get
\begin{align*}
\mu(x+y,x,z)
  &\leq H(2H-1)xz\int_{[0,1]^2}((x\vee z)w_{1})^{2H-2}\text{d}w_{1}\text{d}w_{2}\nonumber\\
	&=    H xz (x\vee z)^{2H-2} = H(x\wedge z) (x\vee z) ^{2H-1}.  
	\end{align*}
Therefore, there exists constant $K\geq0$ such that 
\begin{multline}\label{eq2bxx:13/10/2015}
G_{0,a+b}^{(q)}(a,c)\Indi{(a\wedge c,a\vee c)}(b)\\
\begin{aligned}
	&\leq K(a\wedge c)^{-H(2q+1)+2q-1}(a\vee c)^{-3H-2q+2qH+1}\Indi{(a\wedge c,a\vee c)}(b).
\end{aligned}
\end{multline}
From \eqref{eq2bxx:13/10/2015}, and $H<\frac{3}{4}<\frac{4}{5}\leq \frac{2q}{1+2q}$, it follows that $G_{0,a+b}^{(q)}(a,c)$ is integrable in $\{(a,b,c)\in[0,T]^{3}\ |\ a\wedge c\leq b\leq a\vee c\}$. Finally, by inequalities $\mu\leq a^{H}c^{H}$ and \eqref{eq1:11/11/2015}, we get 
\begin{align}\label{eq2bggffxx:13/10/2015}
G_{0,a+b}^{(q)}(a,c)\Indi{(0,a\wedge c)}(b)
  &\leq a^{-2H}c^{-2H}.
\end{align}
Using \eqref{eq2bggffxx:13/10/2015} as well as condition $H<\frac{3}{4}$, we deduce that $G_{0,a+b}^{(q)}(a,c)$ is integrable in $\{(a,b,c)\in[0,T]^{3}\ |\ b\leq a\wedge c\}$. From the previous analysis it follows that $G_{0,x}^{(q)}(u_{1},u_{2})$ is integrable in $\widetilde{\Tc}$ as required. 
\end{proof}


\end{document}